\documentclass[12pt]{amsart}

\usepackage{amsmath}
\usepackage{amsmath,amsfonts,amssymb,amsthm}
\usepackage{graphicx,color}
\DeclareMathRadical{\sqrtsign}{symbols}{"70}{largesymbols}{"70}
\providecommand{\abs}[1]{\lvert#1\rvert}

\newlength{\figboxwidth}             
\newcommand{\infinity}{\infty}

\def\@ifundefined#1#2#3%
  {\expandafter\ifx\csname#1\endcsname\relax#2\else#3\fi}

\@ifundefined{theoremstyle}{
}{
\theoremstyle{plain} 
}
\newtheorem{theorem}{Theorem}[section]

\newtheorem{proposition}[theorem]{Proposition}
\newtheorem{lemma}[theorem]{Lemma}

\newtheorem{corollary}[theorem]{Corollary}

\@ifundefined{theoremstyle}{
}{
\theoremstyle{definition} 
}
\newtheorem{definition}[theorem]{Definition}

\mathchardef\GG="321D
\newcommand{\mcc}[1]{{}}

\numberwithin{equation}{section}

\title[Concentration of dimension in the Lagrange spectrum]
{Concentration of dimension in extremal points of left-half lines in the Lagrange spectrum}

\author{Carlos Gustavo Moreira}
\address{Carlos Gustavo Moreira: SUSTech International Center for Mathematics, Shenzhen, Guangdong, People’s Republic of China and IMPA, Estrada Dona Castorina 110, 22460-320, Rio de Janeiro, Brazil
}
\email{gugu@impa.br}
\thanks{The first author was partially supported by CNPq and FAPERJ}

\author{Christian Villamil}
\address{Christian Villamil: IMPA, Estrada Dona Castorina 110, 22460-320, Rio de Janeiro, Brazil}
\email{ccsilvav@impa.br}
\thanks{The second author was partially supported by CNPq and Viana's Louis D. prize}


\date{\today}

\begin{document}

\begin{abstract}
We prove that for any $\eta$ that belongs to the closure of the interior of the Markov and Lagrange spectra, the sets  $k^{-1}((-\infty,\eta])$ and $k^{-1}(\eta)$, which are the sets of irrational numbers with best constant of Diophantine approximation bounded by $\eta$ and exactly $\eta$ respectively, have the same Hausdorff dimension. We also show that, as $\eta$ varies in the interior of the spectra, this Hausdorff dimension is a strictly increasing function.
\end{abstract}

\maketitle

\tableofcontents

\section{Introduction}
The classical Lagrange and Markov spectra are closed subsets of the real line related to Diophantine approximations. They arise naturally in the study of rational approximations of irrational numbers and of indefinite binary quadratic forms, respectively. 

Given $\alpha\in\mathbb R\setminus\mathbb Q$, set
\begin{eqnarray*}
	k(\alpha)&=&\sup \left\{k>0:\left|\alpha -\frac{p}{q}\right|<\frac{1}{kq^2} \ \mbox{has infinitely many rational solution} \ \frac{p}{q}\right \}\\&
	=&\limsup_{p\in \mathbb{Z},q\in \mathbb{N}, p,q\to \infty}|q(q\alpha-p)|^{-1}\in \mathbb{R}\cup \{\infty\}
\end{eqnarray*}
for the best constant of Diophantine approximations of $\alpha$.

The {\it classical Lagrange spectrum} is the set 
$$L=\{k(\alpha): \alpha \in \mathbb{R}\setminus\mathbb{Q}, k(\alpha)<\infty\},$$
and the {\it classical Markov spectrum} is the set 
$$M=\left\{\left(\inf\limits_{(x,y)\in\mathbb{Z}^2\setminus\{(0,0)\}} |q(x,y)|\right)^{-1} < \infty: q(x,y)=ax^2+bxy+cy^2, b^2-4ac=1\right\} $$
that consists of the reciprocal of the minimal values over non-trivial integer vectors $(x,y)\in\mathbb{Z}^2\setminus\{(0,0)\}$ of indefinite binary quadratic forms $q(x,y)$ with real coefficients and unit discriminant. 

Perron gave in \cite{P} the following dynamical characterizations of these classical spectra in terms of symbolic dynamical systems: 
Given a bi-infinite sequence $\theta=(\theta_n)_{n\in\mathbb{Z}}\in(\mathbb{N}^*)^{\mathbb{Z}}$, let 
$$\lambda_i(\theta):=[0;a_{i+1},a_{i+2},\dots]+a_i+[0;a_{i-1}, a_{i-2},\dots].$$
If the Markov value $m(\theta)$ of $\theta$ is $m(\theta):=\sup\limits_{i\in\mathbb{Z}} \lambda_i(\theta)$ and the Lagrange value $\ell(\theta)$ is $\ell(\theta):=\limsup \limits_{i\to \infty} \lambda_i(\theta).$ Then the Markov spectrum is the set 
$$M=\{m(\theta)<\infty: \theta\in(\mathbb{N}^*)^{\mathbb{Z}}\}$$
 and the Lagrange spectrum is the set
 $$L=\{\ell(\theta)<\infty: \theta\in(\mathbb{N}^*)^{\mathbb{Z}}\}.$$
 It follows from these characterizations that $M$ and $L$ are closed subsets of $\mathbb R$ and that $L\subset M$.

Markov showed in \cite{M79} that 
$$L\cap (-\infty, 3)=M\cap (-\infty, 3)=\{k_1=\sqrt{5}<k_2=2\sqrt{2}<k_3=\frac{\sqrt{221}}{5}<...\},$$
where $k^2_n\in \mathbb{Q}$ for every $n\in \mathbb{N}$ and $k_n\to 3$ when $n\to \infty$.

M. Hall proved in \cite{Hall} that 
$$C_4+C_4=[\sqrt{2}-1,4(\sqrt{2}-1)],$$
where for each positive integer $N$, $C_N$ is the set of numbers in $[0,1]$ in whose continued fractions the coefficients are bounded by $N$, i.e., $C_N=\{x=[0;a_1,...,a_n,...]\in [0,1]: a_i\le N, \ \forall i\ge 1\}.$ Together with Perron characterizations, this implies that $L$ and $M$ contain the whole half-line $[6,+\infty)$.

Freiman in \cite{F} determined the precise beginning of Hall's ray (the biggest half-line contained in $L$), which is 
$$c_F:=\frac{2221564096 + 283748\sqrt{462}}{491993569}=4.52782956616\dots$$ 

The first author in \cite{M3} proved several results on the geometry of the Markov and Lagrange spectra, for example that the map $d:\mathbb{R} \rightarrow [0,1]$, given by
$$
d(\eta)=HD(L\cap(-\infty,\eta))= HD(M\cap(-\infty,\eta))
$$
is continuous, surjective and such that $d(3)=0$ and $d(\sqrt{12})=1$. Moreover, that
$$d(\eta)=\min \{1,2D(\eta)\},$$
where $D(\eta)=HD(k^{-1}(-\infty,\eta))=HD(k^{-1}(-\infty,\eta])$ is a continuous surjective function from $\mathbb{R}$ to $[0,1).$ And also, the limit
$$\lim_{\eta\rightarrow \infty}HD(k^{-1}(\eta))=1.$$

Recently in \cite{MMPV} was given the estimate 
$$t^*_1:=\sup \{s\in \mathbb{R}:d(s)<1\}=3.334384...$$
In particular, any $t\in \mathbb{R}$ that belongs to the interior of the Markov and Lagrange spectra must satisfy $t>t^*_1.$

Now, let $\varphi:S\rightarrow S$ be a diffeomorphism of a $C^{\infty}$ compact surface $S$ with a mixing horseshoe $\Lambda$ and let $f:S\rightarrow \mathbb{R}$ be a differentiable function. For $x\in S$, following the above characterization of the classical spectra, we define the {\it Lagrange value} of $x$ associated to $f$ and $\varphi$ as being the number $\ell_{\varphi,f}(x)=\limsup_{n\to \infty}f(\varphi^n(x))$ and also the {\it Markov value} of $x$ associated to $f$ and $\varphi$ as the number $m_{\varphi,f}(x)=\sup_{n\in \mathbb{Z}}f(\varphi^n(x))$. 

The sets
$$L_{\varphi,f}(\Lambda)=\{\ell_{\varphi,f}(x):x\in \Lambda\}$$
and
$$M_{\varphi,f}(\Lambda)=\{m_{\varphi,f}(x):x\in \Lambda\}$$
are called {\it Lagrange Spectrum} of $(\varphi,f,\Lambda)$ and {\it Markov Spectrum} of $(\varphi,f, \Lambda)$.

It turns out that dynamical Markov and Lagrange spectra associated to hyperbolic dynamics are closely related to the classical Markov and Lagrange spectra. Several results on the Markov and Lagrange dynamical spectra associated to horseshoes in dimension 2 which are analogous to previously known results on the classical spectra were obtained recently: in \cite{MR} it is shown that typical dynamical spectra associated to horseshoes with Hausdorff dimensions larger than one have non-empty interior (as the classical ones). In \cite{M4} it is shown that typical Markov and Lagrange dynamical spectra associated to horseshoes have the same minimum, which is an isolated point in both spectra and is the image by the function of a periodic point of the horseshoe. 

In \cite{CMM} and \cite{GC1}, in the context of {\it conservative} diffeomorphism it is proven (as a generalization of the results in \cite{CMM}) that for typical choices of the dynamic and of the function, the intersections of the corresponding dynamical Markov and Lagrange spectra with half-lines $(-\infty,t)$ have the same Hausdorff dimensions, and this defines a continuous function of $t$ whose image is $[0,\min \{1,D\}]$, where $D$ is the Hausdorff dimension of the horseshoe. 

For more information and results on classical and dynamical Markov and Lagrange spectra, we refer to the books \cite{CF} and \cite{LMMR}.

In this paper, we use that dynamical Markov and Lagrange spectra associated to conservative horseshoes in surfaces are natural generalizations of the classical Markov and Lagrange spectra. In fact, classical Markov and Lagrange spectra are not compact sets, so they cannot be dynamical spectra associated to horseshoes. However, in \cite{LM2} is showed that, for any $N\ge 2$ with $N\neq 3$, the initial segments of the classical spectra until $\sqrt{N^2+4N}$ (i.e., $M\cap (-\infty,\sqrt{N^2+4N}]$ and $L\cap (-\infty,\sqrt{N^2+4N}]$) coincide with the sets $M(N)$ and $L(N)$, given, in the notation we used in Perron's characterization of $M$ and $L$ by 
$$M(N)=m(\Sigma(N))=\{m(\theta): \theta\in \Sigma(N)\}$$
 and
$$L(N)=\ell(\Sigma(N))=\{\ell(\theta): \theta\in\Sigma(N)\}$$
where $\Sigma(N):=\{1,2,\dots,N \}^{\mathbb{Z}}.$

It is proved also that $M(N)$ and $L(N)$ are dynamical Markov and Lagrange spectra associated to a smooth real function $f$ and to a horseshoe $\Lambda(N)$ defined by a smooth conservative diffeomorphism $\varphi$, and also that they are naturally associated to continued fractions with coefficients bounded by N.

Here we use this relation between classical and dynamical spectra in order to understand better the fractal geometry (Hausdorff dimension) of the preimage of half-lines by the function $k$. We can state our main result as:

\begin{theorem}\label{Teo}
Define $T:=int(L)=int(M)$. For any $\eta\in \overline{T}$, $D(\eta)=HD(k^{-1}(\eta))$ i.e.
$$HD(k^{-1}((-\infty,\eta)))=HD(k^{-1}((-\infty,\eta]))=HD(k^{-1}(\eta)).$$ 
Even more,
\begin{itemize}
    \item if $\eta$ is accumulated from the left by points of $T$, then
    $$D(\eta)>D(t),\ \ \forall t<\eta$$
    \item if $\eta$ is accumulated from the right by points of $T$, then
    $$D(\eta)<D(t),\ \ \forall t>\eta.$$
    
\end{itemize}
In particular, $D|_{X}$ is strictly increasing, where $X$ is $T$ or any interval contained in $\overline{T}$.
\end{theorem}

This result in particular solves a question formulated by Y. Bugeaud and Y. Cheung to the first author: they asked whether there is $T\in\mathbb R$ such that $HD(k^{-1}(t))$ is continuous for $t\in [T,+\infty)$. It follows from Theorem 1.1 that $HD(k^{-1}(t))$ is continuous (and strictly increasing) for $t\in [c_F,+\infty)$. Notice that $c_F$ is minimum with this property.

\section{Preliminaries}

\subsection{Continued fractions and regular Cantor sets} 

The continued fraction expansion of an irrational number $\alpha$ is denoted by 
$$\alpha=[a_0;a_1,a_2,\dots] = a_0 + \frac{1}{a_1+\frac{1}{a_2+\frac{1}{\ddots}}},$$ 
so that the Gauss map $G:(0,1)\to[0,1)$, $G(x)=\dfrac{1}{x}-\left\lfloor \dfrac{1}{x}\right\rfloor$ acts on continued fraction expansions by 
$$G([0;a_1,a_2,\dots]) = [0;a_2,\dots].$$ 

For an irrational number $\alpha=\alpha_0 \in (0,1)$, the continued fraction expansion $\alpha=[0;a_1,\dots]$ is recursively obtained by setting $a_n=\lfloor\alpha_n\rfloor$ and $\alpha_{n+1} = \frac{1}{\alpha_n-a_n} = \frac{1}{G^n(\alpha_0)}$. The rational approximations  
$$\frac{p_n}{q_n}:=[0;a_1,\dots,a_n]\in\mathbb{Q}$$ 
of $\alpha$ satisfy the recurrence relations 
\begin{equation}\label{q_n}
    p_n=a_n p_{n-1}+p_{n-2} \ \mbox{and} \ q_n=a_n q_{n-1}+q_{n-2},\ \ n\geq0
\end{equation}
with the convention that $p_{-2}=q_{-1}=0$ and $p_{-1}=q_{-2}=1$. If $0<a_j\leq N$ for all $j$, it follows that

$$ \frac{p_n}{N+1}\leq p_{n-1}\leq p_n \ \mbox{and} \ \frac{q_n}{N+1}\leq q_{n-1}\leq q_n\ \ n\geq1.$$

Given a finite sequence $(a_1,a_2, \dots, a_n)\in (\mathbb{N}^*)^n$, we define 
$$I(a_1,a_2, \dots, a_n)=\{x\in [0,1]: x=[0;a_1,a_2, \dots, a_n,\alpha_{n+1}],\ \alpha_{n+1}\geq 1 \}$$ 
then by \ref{q_n}, $I(a_1,a_2, \dots, a_n)$ is the interval with extremities $[0;a_1,a_2, \dots, a_n]=\frac{p_n}{q_n}$ and $[0;a_1,a_2, \dots, a_n+1]=\frac{p_n+p_{n-1}}{q_n+q_{n-1}}$ and so

$$\abs{I(a_1,a_2, \dots, a_n)}=\left |\frac{p_n}{q_n}-\frac{p_n+p_{n-1}}{q_n+q_{n-1}}\right |=\frac{1}{q_n(q_n+q_{n-1})},$$
because $p_n q_{n-1}-p_{n-1} q_n=(-1)^{n-1}.$

Also, for $(a_0,a_1, \dots, a_n)\in (\mathbb{N}^*)^{n+1}$ we set 
$$I(a_0;a_1, \dots, a_n)=\{x\in [0,1]: x=[a_0;a_1,a_2, \dots, a_n,\alpha_{n+1}],\ \alpha_{n+1}\geq 1 \},$$
clearly as $I(a_0;a_1, \dots, a_n)=a_0+I(a_1,a_2, \dots, a_n)$, we have 
\begin{equation}\label{intervals}
    \abs{I(a_0;a_1, \dots, a_n)}=\abs{I(a_1,a_2, \dots, a_n)}.
\end{equation}

An elementary result for comparing continued fractions is the following lemma

\begin{lemma}\label{leminha}
    Let $\alpha=[a_0;a_1,\dots, a_n, a_{n+1},\dots]$ and $\tilde{\alpha}=[a_0;a_1,\dots, a_n, b_{n+1},\dots]$, then:
    \begin{itemize}
        \item $\abs{\alpha-\tilde{\alpha}}<1/2^{n-1},$ 
        
        \item If $a_{n+1}\neq b_{n+1}$, $\alpha>\tilde{\alpha}$ if and only if $(-1)^{n+1}(a_{n+1}-b_{n+1})>0.$
        
    \end{itemize}
\end{lemma}
 The next lemma is from \cite{M3} (see lemma A.1)

 \begin{lemma} \label{gugu1}
     If $a_0,a_1,a_2\dots, a_n, a_{n+1}, \dots$ and $b_{n+1},b_{n+2}, \dots$ are positive integers bounded by $N\in \mathbb{N}$ and $a_{n+1}\neq b_{n+1}$
then
\begin{eqnarray*}
 \abs{[a_0;a_1,a_2\dots, a_n, a_{n+1}, \dots]-[a_0;a_1,a_2\dots,a_n,b_{n+1}, \dots]}
&>& c(N)/q^2_{n-1}  \\ &>& c(N)\abs{I(a_1,a_2, \dots, a_n)}   
\end{eqnarray*}
for some positive constant $c(N).$
 \end{lemma}

For the sequel, the following application of lemma \ref{leminha} will also be useful 

\begin{lemma}\label{lemao}
Given $R, N\in \mathbb{N}$, let $\beta^1,\beta^2,\beta^3\in \Sigma(N)^+:=\{1,2,\dots,N \}^{\mathbb{N}}$ such that $[0;\beta^1]<[0;\beta^2]<[0;\beta^3]$. If for two sequences $\alpha=(\alpha_n)_{n\in \mathbb{Z}}\ \mbox{and}\ \tilde{\alpha}=(\tilde{\alpha}_n)_{n\in \mathbb{Z}}\ \mbox{in}\ \Sigma(N)$ it is true that  $\alpha_{0},\dots ,\alpha_{2R+1}=\tilde{\alpha}_{0},\dots ,\tilde{\alpha}_{2R+1}$, then, for all $j\leq 2R+1$ we have
\begin{eqnarray*}
    \lambda_0(\sigma^j(\dots,\alpha_{-2},\alpha_{-1};\alpha_{0},\dots ,\alpha_{2R+1},\beta^2))< \max \{m(\dots,\alpha_{-2},\alpha_{-1};\alpha_{0},\dots ,\alpha_{2R+1},\beta^1), \\ m(\dots ,\tilde{\alpha}_{-2},\tilde{\alpha}_{-1};\tilde{\alpha}_{0},\dots ,\tilde{\alpha}_{2R+1},\beta^3)\}+1/2^{R-1}.\end{eqnarray*}
    
\end{lemma}

\begin{proof}
   It is just an application of lemma \ref{leminha}. Indeed, for $j \leq R+1$
   \begin{eqnarray*}
     \lambda_0(\sigma^j(\dots,\alpha_{-1};\alpha_{0},\dots ,\alpha_{2R+1},\beta^2))< \lambda_0(\sigma^j(\dots,\alpha_{-1};\alpha_{0},\dots ,\alpha_{2R+1},\beta^1))+ 1/2^{R-1} \\ \leq \max \{m(\dots,\alpha_{-1};\alpha_{0},\dots ,\alpha_{2R+1},\beta^1), m(\dots ,\tilde{\alpha}_{-1};\tilde{\alpha}_{0},\dots ,\tilde{\alpha}_{2R+1},\beta^3)\} +1/2^{R-1}.  
   \end{eqnarray*}
For $R+1<j\leq 2R+1$, if $[\alpha_{j};\dots ,\alpha_{2R+1},\beta^2]<[\tilde{\alpha}_{j};\dots ,\tilde{\alpha}_{2R+1},\beta^3]$
\begin{eqnarray*}
    \lambda_0(\sigma^j(\dots,\alpha_{-2},\alpha_{-1};\alpha_{0},\dots ,\alpha_{2R+1},\beta^2))< \lambda_0(\sigma^j(\dots ,\tilde{\alpha}_{-1};\tilde{\alpha}_{0},\dots ,\tilde{\alpha}_{2R+1},\beta^3)) + 1/2^{R}\\  \leq \max \{m(\dots,\alpha_{-1};\alpha_{0},\dots ,\alpha_{2R+1},\beta^1), m(\dots ,\tilde{\alpha}_{-1};\tilde{\alpha}_{0},\dots ,\tilde{\alpha}_{2R+1},\beta^3)\}+1/2^{R}.
\end{eqnarray*}
And for $R+1<j\leq 2R+1$, if $[\alpha_{j};\dots ,\alpha_{2R+1},\beta^2]<[\alpha_{j};\dots ,\alpha_{2R+1},\beta^1]$
\begin{eqnarray*}
   \lambda_0(\sigma^j(\dots,\alpha_{-1};\alpha_{0},\dots ,\alpha_{2R+1},\beta^2))< \lambda_0(\sigma^j(\dots ,\alpha_{-1};\alpha_{0},\dots ,\alpha_{2R+1},\beta^1))\\ \leq \max \{m(\dots,\alpha_{-1};\alpha_{0},\dots ,\alpha_{2R+1},\beta^1), m(\dots ,\tilde{\alpha}_{-1};\tilde{\alpha}_{0},\dots ,\tilde{\alpha}_{2R+1},\beta^1)\}. 
\end{eqnarray*}
Then we have proved the result.
\end{proof}
 We end this subsection with one definition 
\begin{definition}\label{cantor}
	A set $K\subset \mathbb{R}$ is called a $C^{1+\alpha}$-{\it regular Cantor set}, $\alpha > 0$, if there exists a collection $\mathcal{P}=\{I_1,I_2,...,I_r\}$ of compacts intervals and a $C^{1+\alpha}$-expanding map $\psi$, defined in a neighbourhood of $\displaystyle \cup_{1\leq j\leq r}I_j$ such that
	
	\begin{enumerate}
		\item[(i)] $K\subset \cup_{1\leq j\leq r}I_j$ and $\cup_{1\leq j\leq r}\partial I_j\subset K$,
		
		\item[(ii)] For every $1\leq j\leq r$ we have that $\psi(I_j)$ is the convex hull of a union of $I_r$'s, for $l$ sufficiently large $\psi^l(K\cap I_j)=K$ and $$K=\bigcap_{n\geq 0}\psi^{-n}(\bigcup_{1\leq j\leq r}I_j).$$
	\end{enumerate}
More precisely, we also say that the triple $(K, \mathcal{P}, \psi)$ is a $C^{1+\alpha}$-regular Cantor set.
\end{definition}

For example, in our context of sets of continued fractions. Let, as before, $G$ be the Gauss map and $C_N=\{x=[0;a_1,a_2,...]: a_i\le N, \forall i\ge 1\}$. Then, 
	$$C_N=\bigcap_{n\ge 0}G^{-n}(I_N\cup... \cup I_1),$$
where $I_j=[a_j,b_j]$ and $a_j=[0;j,\overline{1,N}]$ and $b_j=[0;j,\overline{N,1}].$ That is, $C_N$ is a regular Cantor set. 

In a similar way, the set $\tilde{C}_N=\{1,2,...,N\}+C_N$ with the map $G^{'}$ given for $a\in \{1,\dots,N\}$ and $x=[0;a_1,a_2,...]\in C_N$ by
$$G^{'}(a+x)=\frac{1}{a+x-\lfloor a+x\rfloor}=\frac{1}{x}=G(x)+ \left\lfloor \frac{1}{x}\right\rfloor=G(x)+a_1(x)$$
is also a regular Cantor set.
\subsection{Results on Dynamical Markov and Lagrange spectra}

Let $\varphi:S\rightarrow S$ be a diffeomorphism of a $C^{\infty}$ compact surface $S$ with a mixing horseshoe $\Lambda$ and let $f:S\rightarrow \mathbb{R}$ be a differentiable function. Fix a Markov partition $\{R_a\}_{a\in \mathcal{A}}$ with sufficiently small diameter consisting of rectangles $R_a \sim I_a^u \times I_a^s$ delimited by compact pieces $I_a^s$, $I_a^u$, of stable and unstable manifolds of certain points of $\Lambda$ (see \cite{PT93} theorem 2, page 172). The set $\mathcal{B}\subset \mathcal{A}^{2}$ of admissible transitions consist of pairs $(a,b)$ such that $\varphi(R_a)\cap R_{b}\neq \emptyset$; so, we can define the transition matrix $B$ by
$$b_{ab}=1 \ \ \text{if} \ \  \varphi(R_a)\cap R_b\neq \emptyset \ \ \text{and}  \ b_{ab}=0 \ \ \ \text{otherwise, for $(a,b)\in \mathcal{A}^{2}$.}$$

Let $\Sigma_{\mathcal{A}}=\left\{\underline{a}=(a_{n})_{n\in \mathbb{Z}}:a_{n}\in \mathcal{A} \ \text{for all} \ n\in \mathbb{Z}\right\}$ and consider the homeomorphism of $\Sigma_{\mathcal{A}}$, the shift, $\sigma:\Sigma_{\mathcal{A}}\to\Sigma_{\mathcal{A}}$ defined by $\sigma(\underline{a})_{n}=a_{n+1}$. Let $\Sigma_{\mathcal{B}}=\left\{\underline{a}\in \Sigma_{\mathcal{A}}:b_{a_{n}a_{n+1}}=1\right\}$, this set is closed and $\sigma$-invariant subspace of $\Sigma_{\mathcal{A}}$. Still denote by $\sigma$ the restriction of $\sigma$ to $\Sigma_{\mathcal{B}}$, the pair $(\Sigma_{\mathcal{B}},\sigma)$ is a subshift of finite type, see \cite{{Shub}} chapter 10. The dynamics of $\varphi$ on $\Lambda$ is topologically conjugate to the sub-shift $\Sigma_{\mathcal{B}}$, namely, there is a homeomorphism $\Pi: \Lambda \to \Sigma_{\mathcal{B}}$ such that $\varphi\circ \Pi=\Pi\circ \sigma$.

Recall that the stable and unstable manifolds of $\Lambda$ can be extended to locally invariant $C^{1+\alpha}$ foliations in a neighborhood of $\Lambda$ for some $\alpha>0$. Using these foliations it is possible define projections $\pi^u_a: R_a\rightarrow \{i^u_a\} \times I^s_a$ and $\pi^s_a: R_a\rightarrow I^u_a \times \{i^s_a\}$ of the rectangles into the connected components $\{i^u_a\} \times I^s_a$ and $I^u_a \times \{i^s_a\}$ of the stable and unstable boundaries of $R_a$, where $i^u_a\in \partial I^u_a$ and $i^s_a\in \partial I^s_a$ are fixed arbitrarily. In this way, we have the unstable and stable Cantor sets
$$K^u=\bigcup_{a\in \mathcal{A}}\pi^s_a(\Lambda\cap R_a) \ \mbox{and} \ K^s=\bigcup_{a\in \mathcal{A}}\pi^u_a(\Lambda\cap R_a).$$

In fact $K^u$ and $K^s$ are $C^{1+\alpha}$ dynamically defined, associated to some expanding maps $\psi_s$ and $\psi_u$ defined in the following way: If $y\in R_{a_1}\cap \varphi(R_{a_0})$ we put
$$\psi_s(\pi^u_{a_1}(y))=\pi^u_{a_0}(\varphi^{-1}(y))$$
and if $z\in R_{a_0}\cap \varphi^{-1}(R_{a_1})$ we put
$$\psi_u(\pi^s_{a_0}(z))=\pi^s_{a_1}(\varphi(z)).$$

The stable and unstable Cantor sets, $K^s$ and $K^u$, respectively, are closely related to the fractal geometry of the horseshoe $\Lambda$. For instance, it is well-known that, 
$$HD(\Lambda)=HD(K^s)+HD(K^u)$$
and that in the conservative case 
$$HD(K^s)=HD(K^u).$$

The study of the intersection of the spectra with half-lines is related to the study of fractal dimensions of the set
$$\Lambda_t=\bigcap\limits_{n\in\mathbb{Z}}\varphi^{-n}(\{y\in\Lambda: f(y)\leq t\}) = \{x\in\Lambda: m_{\varphi, f}(x)=\sup\limits_{n\in\mathbb{Z}}f(\varphi^n(x))\leq t\}$$ 
for $t\in\mathbb{R}$. Following this, we also consider the subsets $\Lambda_t$ through its projections on the stable and unstable Cantor sets of $\Lambda$
$$K^u_t=\bigcup_{a\in \mathcal{A}} \pi^s_a(\Lambda_t\cap R_a) \ \mbox{and} \ K^s_t=\bigcup_{a\in \mathcal{A}}\pi^u_a(\Lambda_t\cap R_a).$$

In \cite{GC1} is showed the following result, extending a previous result from \cite{CMM}:

\begin{theorem}\label{janelas}
   Let $r\geq 2$ and $\varphi\in \text{Diff}^{2}(S)$ a conservative diffeomorphism preserving a smooth area form $\omega$ and take $\Lambda$ a mixing horseshoe of $\varphi$. If $f\in C^r(S,\mathbb{R})$ satisfies that  $\forall \ z\in \Lambda, \ \nabla f(z)\neq 0$, then the functions
	$$t\mapsto HD(K^u_t) \ \mbox{and} \ t\mapsto HD(K^s_t)$$
are equal and continuous. 
Even more, one has
$$HD(\Lambda_t)=2HD(K^u_t).$$
\end{theorem}

\subsection{The horseshoe $\Lambda(N)$}

Given an integer $N\ge 2$, write
$$\Lambda(N)=C_N\times \tilde{C}_N.$$
If $x=[0;a_1,a_2,...]$ and $y=[a_0;a_{-1},a_{-2},...]$ then we consider $\varphi:\Lambda(N) \rightarrow \Lambda(N)$ given by
	\begin{eqnarray*}\label{dif}
	    \varphi(x,y)&=&(G(x),a_1+1/y)\\
     &=&([0;a_2,a_3,...],a_1+[0;a_0,a_{-1},...]).
	\end{eqnarray*}
Also, equip $\Lambda(N)$ with the real map $f(x,y)=x+y$. We note that $\varphi$ can be extended to a $C^{\infty}$-diffeomorphism on a diffeomorphic copy of the 2-dimensional sphere $\mathbb{S}^2$.

Notice also that $\varphi$ is conjugated to the restriction to $C_N\times C_N$ of the map $\psi:(0,1)\times(0,1)\to [0,1)\times(0,1)$ given by 
$$\psi(x,y)=\left(G(x),\frac1{y+\lfloor 1/x\rfloor}\right)$$ and following \cite{Ar} and \cite{S.ITO} we know that $\psi$ has an invariant measure equivalent to the Lebesgue measure. In particular, $\varphi$ also has an invariant measure equivalent to the Lebesgue measure and then $\varphi$ is conservative. 

Indeed, if $\mathcal{S}=\{(x,y)\in {\mathbb R}^2|0<x<1, 0<y<1/(1+x)\}$ and $T:\mathcal{S}\to\mathcal{S}$ is given by
$$T(x,y)=(G(x),x-x^2 y),$$
then $T$ preserves the Lebesgue measure in the plane.
If $h:\mathcal{S}\to [0,1)\times(0,1)$ is given by $h(x,y)=(x,y/(1-xy))$ then $h$ is a conjugation between $T$ and $\psi$ (and thus $\psi$ preserves the smooth measure $h_*$(Leb)).

For $\Lambda(N)$ we have the Markov partition $\{R_a\}_{a\in \mathcal{A}}$ where $\mathcal{A}=\{1,2, \dots,N\}$ and $R_a$ is such that $R_a\cap \Lambda(N)=C_N\times(C_N+a)=C_N\times C_N+(0,a)$. By definition, $\varphi$ expands in the $x$-direction and contracts in the $y$-direction. Therefore, for $(x,y)\in R_a$ we can set $\pi^s_a(x,y)=(x,a+[0;\overline{N,1}])$ and 
\begin{eqnarray*}
\psi_u(x,a+[0;\overline{N,1}])&=&\pi ^s_{a_1(x)}(\phi(x,a+[0;\overline{N,1}]))\\
&=& \pi ^s_{a_1(x)}(G(x),a_1(x)+1/(a+[0;\overline{N,1}]))\\
&=&(G(x),a_1(x)+[0;\overline{N,1}])
	\end{eqnarray*}
thus we can identify $(K^u(\Lambda(N)),\psi_u)$ with $(\tilde{C}_N,G^{'})$. A similar identification can be made for $(K^s(\Lambda(N)),\psi_s)$.

One has that $\varphi|_{\Lambda_N}$ is topologically conjugated to $\sigma:\Sigma(N)\rightarrow \Sigma(N)$ (via a map $\Pi: \Lambda(N) \to \Sigma(N)$), and that in sequences, $f$ becomes $\tilde{f}:\Sigma(N)\rightarrow \mathbb{R}$ given by 
$$\tilde{f}(\theta)=[0;a_1(\theta),a_2(\theta),...]+a_0(\theta)+[0;a_{-1}(\theta),a_{-2}(\theta),...]=\lambda_0(\theta),$$ where $\theta=(a_i(\theta))_{i\in \mathbb{Z}},$ and so
$$L_{\varphi,f}(\Lambda(N))=L(N)\ \ \mbox{and}\ \ M_{\varphi,f}(\Lambda(N))=M(N).$$

In this context, let $\alpha=(a_{s_1},a_{s_1+1},...,a_{s_2})\in \mathcal{A}^{s_2-s_1+1}$ any word where $s_1, s_2 \in \mathbb{Z} , \ s_1 < s_2$ and fix $s_1\le s\le s_2$. Define then
	$$R(\alpha;s):=\bigcap_{m=s_1-s}^{s_2-s} \varphi^{-m}(R_{a_{m+s}}).$$
Note that if $x\in R(\alpha;s)\cap \Lambda(N)$ then the symbolic representation of $x$ is in the way $(...a_{s_1}...a_{s-1};a_{s},a_{s+1}...a_{s_2}...)$ where on the right of the ; is the $0$-th position. 

Finally, let us consider $A_N=[0;\overline{N,1}]$ and $B_N=[0;\overline{1,N}]$. As 
$$NA_N+ A_NB_N=1\ \mbox{and}\  B_N+B_NA_N=1,$$ 
we have $A_N=\dfrac{B_N}{N}.$  
Thus $B_N=\frac{-N+\sqrt{N^2+4N}}{2} \ \mbox{,} \ A_N=\frac{-N+\sqrt{N^2+4N}}{2N}$ and then
$$\max f|_{\Lambda(N)}=2B_N+N=\sqrt{N^2+4N},\ \min f|_{\Lambda(N)}=2A_N+1=\frac{\sqrt{N^2+4N}}{N}.$$

\section{Proof of the main theorem}

\subsection{Connection of subhorseshoes}
For the next, it will be useful to give the following definition
\begin{definition}\label{conection of horseshoes}
Given $\Lambda^1$ and $\Lambda^2$ subhorseshoes of a horseshoe $\Lambda$ and $t\in \mathbb{R}$, we said that $\Lambda^1$ \emph{connects} with $\Lambda^2$ or that $\Lambda^1$ and $\Lambda^2$ \emph{connect} before $t$ if there exist a subhorseshoe $\tilde{\Lambda}\subset \Lambda$ and some $q< t$ with $\Lambda^1 \cup \Lambda^2 \subset \tilde{\Lambda}\subset \Lambda_q$.
\end{definition}

\begin{lemma}\label{connection}
Suppose $\Lambda^1$ and $\Lambda^2$ are subhorseshoes of $\Lambda$ and for some $x,y \in \Lambda$ we have $x\in W^u(\Lambda^1)\cap W^s(\Lambda^2)$ and $y\in W^u(\Lambda^2)\cap W^s(\Lambda^1)$. If for some $t\in \mathbb{R}$, \ it is true that 
$$\Lambda^1 \cup \Lambda^2 \cup \mathcal{O}(x) \cup \mathcal{O}(y) \subset \Lambda_t,$$ then for every $\epsilon >0$,\ $\Lambda^1$ and $\Lambda^2$ connect before $t+\epsilon$. 
\end{lemma}

\begin{proof}
Take a Markov partition $\mathcal{P}$ for $\Lambda$ with diameter small enough such that $ \max f|_{\bigcup\limits_{P\in \mathcal{R}} P}< t+\epsilon,$
where $\mathcal{R}=\{P\in \mathcal{P}: P\cap (\Lambda^1 \cup \Lambda^2 \cup \mathcal{O}(x) \cup \mathcal{O}(y))\neq \emptyset \}$ and consider
$$\Lambda_{\mathcal{R}}=\bigcap \limits_{n \in \mathbb{Z}} \varphi ^{-n}(\bigcup \limits_{P\in \mathcal{R}} P).$$ Evidently $\Lambda^1 \cup \Lambda^2 \cup \mathcal{O}(x) \cup \mathcal{O}(y)\subset \Lambda_{\mathcal{R}}\subset \Lambda_{t+\epsilon}$, then the lemma will be proved if we show that $\Lambda^1$ and $\Lambda^2$ form part of the same transitive component of $\Lambda_{\mathcal{R}}$.

Let $x_1\in \Lambda^1$, \ $x_2\in \Lambda^2$ and $\rho_1,\rho_2>0$. Take $$\eta=\frac{1}{2}\min \{\rho_1, \rho_2, \min \{d(P,Q):P,Q\in \mathcal{R}\  \mbox{and}\ P\neq Q \} \}.$$ By the shadowing lemma there exist $0<\delta \leq \eta$ such that every $\delta$-pseudo orbit of $\Lambda$ is $\eta$-shadowed by the orbit of some element of $\Lambda$.

On the other hand, as $\varphi|_{\Lambda^1}$ is transitive and $x\in W^u(\Lambda^1)$ there exist $y_1\in \Lambda^1 \cap B(x_1,\delta)$ and $N_1, M_1\in \mathbb{N}$ such that $d(\varphi^{M_1}(y_1),\varphi^{-N_1}(x))<\delta$ and analogously  as $\varphi|_{\Lambda^2}$ is transitive and $x\in W^s(\Lambda^2)$ there exist $y_2\in \Lambda^2$ and $N_2, M_2\in \mathbb{N}$ such that $d(\varphi^{N_2}(x),y_2)<\delta$ and $d(x_2,\varphi^{M_2}(y_2))<\delta$. Define then the $\delta$-pseudo orbit:
$$\dots ,\varphi^{-1}(y_1);y_1,\varphi(y_1), \dots,  \varphi^{M_1-1}(y_1), \varphi^{-N_1}(x),\dots, \varphi^{N_2-1}(x), y_2, \varphi(y_2), \dots$$

Then there exists $w\in \Lambda$ that $\eta$-shadowed that orbit. Moreover as the $\delta$-pseudo orbit have all its terms in $\bigcup \limits_{P\in \mathcal{R}} P$ and $\eta \leq \frac{1}{2} \min \{d(P,Q):P,Q\in \mathcal{R}\  \mbox{and}\ P\neq Q \} $ we have also $\mathcal{O}(w)\subset \bigcup \limits_{P\in \mathcal{R}} P\ ;$ that is, $w\in \Lambda_{\mathcal{R}}$ and furthermore 
$$w\in B(x_1,\rho_1) \quad  \mbox{and} \quad \varphi^{M_1+N_1-1+N_2+M_2}(w)\in B(x_2,\rho_2).$$

The proof that there exists $w\in B(x_2, \rho_2)$ and $M\in \mathbb{N}$ such that $\varphi^M(w)\in B(x_1, \rho_1)$ is analog.
\end{proof}
 \begin{corollary}\label{connection2}
 Suppose $\Lambda^1$ and $\Lambda^2$ are subhorseshoes of $\Lambda$ with $\Lambda^1 \cup \Lambda^2 \subset \Lambda_t$ for some $t\in \mathbb{R}$. If $\Lambda^1 \cap \Lambda^2\neq \emptyset$, then for every $\epsilon >0$, \ $\Lambda^1$ and $\Lambda^2$ connects before $t+\epsilon$. 
 \end{corollary}
 \begin{proof}
 If $\Lambda^1 \cap \Lambda^2\neq \emptyset$,\ then every $w\in \Lambda^1 \cap \Lambda^2$ satisfies $w\in W^u(\Lambda^1)\cap W^s(\Lambda^2)$ and $w\in W^u(\Lambda^2)\cap W^s(\Lambda^1)$ and then we have the conclusion.
 \end{proof}
 \begin{corollary}\label{connection3}
 Let $\Lambda^1$,\ $\Lambda^2$ and $\Lambda^3$ subhorseshoes of $\Lambda$ and $t\in \mathbb{R}$. If $\Lambda^1$\ connects with $\Lambda^2$ before $t$ and $\Lambda^2$\ connects with $\Lambda^3$ before $t$. Then also $\Lambda^1$\ connects with $\Lambda^3$ before $t$.
 \end{corollary}
 \begin{proof}
By hypothesis we have two subhorseshoes $\Lambda^{1,2}$ and $\Lambda^{2,3}$ and $q_1,q_2<t$ with$$\Lambda^1 \cup \Lambda^2 \subset \Lambda^{1,2}\subset \Lambda_{q_1} \ \mbox{and }\ \Lambda^2 \cup \Lambda^3 \subset \Lambda^{2,3}\subset \Lambda_{q_2}.$$
Applying corollary \ref{connection2} to $\Lambda^{1,2}$ and $\Lambda^{2,3}$,\ with $\tilde{t}=\max \{q_1,q_2\}$ and $\epsilon=t-\tilde{t}$ we have the result.
\end{proof}

\subsection{Dimension estimates}\label{section3}
Fix an integer $m\geq 1$ and consider the horseshoe
$$\Lambda:=\Lambda(m+3)=C(m+3)\times\tilde{C}(m+3)$$
equipped with the diffeomorphism $\varphi$ and the map $f$ given in the previous section. 
Given $\epsilon>0$, we can take $\ell(\epsilon)\in \mathbb{N}$ sufficiently large such that if $\alpha=(a_0, a_{1} \cdots, a_{2\ell(\epsilon)})\in \{1,2, \cdots, m+3\}^{2\ell(\epsilon)+1}$
and $x, y\in R(\alpha;\ell(\epsilon))\ \mbox{then}\ \abs{f(x)-f(y)}<\epsilon/4.$ 

Consider 
$$\eta \in (m+1+ [0;\overline{1}]+[0;1,m+2,\overline{1,m+3}],m+4)\cap \overline{T}$$0
which is accumulated from the left by points of $T$.

Given $t\in(m+1+ [0;\overline{1}]+[0;1,m+2,\overline{1,m+3}],\eta)\cap T$ and $0<\epsilon<\eta-t$, let
$$C(t,\epsilon)=\{\alpha=(a_0, a_{1} \cdots, a_{2\ell(\epsilon)})\in \{1,2, \cdots, m+3\}^{2\ell(\epsilon)+1}:R(\alpha;\ell(\epsilon))\cap \Lambda_{t+\epsilon/4}\neq \emptyset \}.$$

Define
$$P(t,\epsilon):=\bigcap \limits_{n \in \mathbb{Z}} \varphi ^{-n}(\bigcup \limits_{\alpha \in C(t,\epsilon)}  R(\alpha;\ell(\epsilon))).$$
Note that by construction, $\Lambda_{t+\epsilon/4}\subset P(t,\epsilon)\subset \Lambda_{t+\epsilon/2}$ and being $P(t,\epsilon)$ a hyperbolic set of finite type (see Appendix A for the corresponding definitions and results), it admits a decomposition 
$$P(t,\epsilon)=\bigcup \limits_{x\in \mathcal{X}} \tilde{\Lambda}_x $$
 where $\mathcal{X}$ is a finite index set and for $x\in \mathcal{X}$,\ $\tilde{\Lambda}_i$ is a subhorseshoe or a transient set i.e a set of the form $\tau=\{x\in M: \alpha(x)\subset \tilde{\Lambda}_{i_1} \  \mbox{and } \ \omega(x)\subset \tilde{\Lambda}_{i_2}\}$ where $\tilde{\Lambda}_{i_1}$ and $\tilde{\Lambda}_{i_2}$ with $i_1, i_2 \in \mathcal{X}$ are subhorseshoes. 
 
As for every transient $\tau$ set as before, we have
$$HD(\tau)=HD(K^s(\tilde{\Lambda}_{i_1}))+HD(K^u(\tilde{\Lambda}_{i_2}))$$
and for every subhorseshoe $\tilde{\Lambda}_{i}$, being  $\varphi$ conservative, one has
$$HD(\tilde{\Lambda}_{i})=HD(K^s(\tilde{\Lambda}_{i}))+HD(K^u(\tilde{\Lambda}_{i}))=2HD(K^u(\tilde{\Lambda}_{i}))$$
therefore
\begin{equation}
  HD(P(t,\epsilon))=\max\limits_{x\in \mathcal{X}} HD(\tilde{\Lambda}_{x})=\max \limits_{\substack{x\in \mathcal{X}: \ \tilde{\Lambda}_x \ is\\ subhorseshoe }}HD(\tilde{\Lambda}_{x}).  
\end{equation}

Now, in \cite{M3} was proved for $s\leq \max f|_{\Lambda}$ that
 $$D(s)=HD(k^{-1}(-\infty,s])=HD(K^u_s)$$
and by theorem \ref{janelas}, we have 
$$HD(K^u_s)=\frac{1}{2}HD(\Lambda_{s}).$$
Then, for some $x\in \mathcal{X}$, $HD(\tilde{\Lambda}_x)\geq1$ because $\Lambda_t\subset P(t,\epsilon)$ and 
$$t^*_1=\sup \{s\in \mathbb{R}:\min \{1,HD(\Lambda_s)\}<1\}=\sup \{s\in \mathbb{R}:HD(\Lambda_s)<1\}<t.$$

We will show that any subhorseshoe contained in $P(t,\epsilon)$ with Hausdorff dimension greater or equal than $1$ connects with the fixed orbit $\xi$, given by the kneading sequence $(1)_{i\in \mathbb{Z}}$, before any time bigger than $t+\epsilon/2$. To do that, take any $\delta>0$ and write
$$\tilde{P}(t,\epsilon)=\bigcup\limits_{\substack{x\in \mathcal{X}: \ \tilde{\Lambda}_x \ is\\ subhorseshoe }}\tilde{\Lambda}_x= \bigcup \limits_{i\in \mathcal{I}} \tilde{\Lambda}_i \cup \bigcup \limits_{i\in \mathcal{J}} \tilde{\Lambda}_j$$ where
 $$\mathcal{I}=\{i\in \mathcal{X}: \tilde{\Lambda}_i \ \mbox{is a subhorseshoe and it connects with}\ \xi\ \mbox{before}\ t+\epsilon/2+\delta \}$$
and 
 $$\mathcal{J}=\{j\in \mathcal{X}: \tilde{\Lambda}_j \ \mbox{is a subhorseshoe and it doesn't connect with}\ \xi\ \mbox{before}\ t+\epsilon/2+\delta \}.$$


By proposition \ref{connection}, given $j\in \mathcal{J}$ as $\tilde{\Lambda}_j\cup \xi \subset \Lambda_{t+\epsilon/2}$  we cannot have at the same time the existence of two points $x\in W^u(\tilde{\Lambda}_j)\cap W^s(\xi)$ and $y\in W^u(\xi)\cap W^s(\tilde{\Lambda}_j)$ such that $\mathcal{O}(x) \cup \mathcal{O}(y) \subset \Lambda_{t+\epsilon/2+\delta/2}$. Without loss of generality suppose that there is no $x\in W^u(\tilde{\Lambda}_j)\cap W^s(\xi)$ with $m_{\varphi,f}(x)\leq t+\epsilon/2+\delta/2$ (the argument for the other case is similar). We will show that this condition forces the possible letters that may appear in the sequences that determine the unstable Cantor set of $\tilde{\Lambda}_j$.

Let us begin fixing $R\in \mathbb{N}$ large enough such that $1/2^{R-1}<\delta/2$ and consider the set $\mathcal{C}_{2R+1}=\{I(a_0;a_1, \dots, a_{2R+1}): I(a_0;a_1, \dots, a_{2R+1})\cap K^u(\tilde{\Lambda}_j)\neq \emptyset\}$; clearly $\mathcal{C}_{2R+1}$ is a covering of $K^u(\tilde{\Lambda}_j)$. We will give a mechanism to construct coverings $\mathcal{C}_k$
with $k\geq 2R+1$ that can be used to \emph{efficiently} cover $K^u(\tilde{\Lambda}_j)$ as $k$ goes to infinity.

Indeed, if for some $k\geq 2R+1$, and $I(a_0; a_1, \dots,a_k)\in \mathcal{C}_k$,  $(a_0, a_1, \dots,a_k)$ has continuations with forced first letter, that is, for every $\alpha=(\alpha_n)_{n\in \mathbb{Z}}\in \Pi(\tilde{\Lambda}_j)$ with $\alpha_0,\alpha_1, \dots, \alpha_k=a_0,a_1, \dots, a_k$ one has $\alpha_{k+1}=a_{k+1}$ for some fixed $a_{k+1}$, then we can refine the original cover $\mathcal{C}_k$, by replacing the interval $I(a_0;a_1, \dots, a_k)$ with the interval $I(a_0;a_1, \dots, a_k, a_{k+1})$.

On the other hand, suppose that $(a_0, a_1, \dots,a_k)$ has two continuations with different initial letter, say $ \gamma_{k+1}=(a_{k+1}, a_{k+2},\dots)$ and $\beta_{k+1}=(a^*_{k+1}, a^*_{k+2},\dots)$ with $a_{k+1} \neq a^*_{k+1}$. Take $\alpha=(\alpha_n)_{n\in \mathbb{Z}}\in \Pi(\tilde{\Lambda}_j)$ and $\tilde{\alpha}=(\tilde{\alpha}_n)_{n\in \mathbb{Z}}\in \Pi(\tilde{\Lambda}_j)$, such that $\alpha=(\dots,\alpha_{-2},\alpha_{-1};a_0, a_1, \dots,a_k,\gamma_{k+1})$ and $\tilde{\alpha}=(\dots,\tilde{\alpha}_{-2},\tilde{\alpha}_{-1};a_0, a_1, \dots,a_k,\beta_{k+1})$.  If \\ $a_{k+1}=i$ then, necessarily either $a^*_{k+1}=i+1$ or $a^*_{k+1}=i-1$ because if for example $a_{k+1}+1<a^*_{k+1}$ we can set $s=a_{k+1}+1$ and therefore by lemma \ref{lemao} as $[0;\beta_{k+1}]<[0;s,\overline{1}]<[0;\gamma_{k+1}]$, we would have for all $j\leq k$
\begin{eqnarray*}
    \lambda_0(\sigma^j(\dots,\tilde{\alpha}_{-2},\tilde{\alpha}_{-1};\tilde{\alpha}_{0},\dots ,\tilde{\alpha}_{k},s,\overline{1}))&\leq&\max \{m(\dots, \alpha_{-1};\alpha_{0},\dots ,\alpha_{k},\gamma_{k+1}),\\&&  m(\dots ,\tilde{\alpha}_{-1};\tilde{\alpha}_{0},\dots ,\tilde{\alpha}_{k},\beta_{k+1})\}+1/2^{R-1} \\ &<& t+\epsilon/2+\delta/2.
\end{eqnarray*}
For $j=k+1$,
\begin{eqnarray*}
  \lambda_0(\sigma^j(\dots,\tilde{\alpha}_{-2},\tilde{\alpha}_{-1};\tilde{\alpha}_{0},\dots ,\tilde{\alpha}_{k},s,\overline{1}))&=& [0;\tilde{\alpha}_{k},\dots, \tilde{\alpha}_{0},\tilde{\alpha}_{-1}, \dots ]+s+[0;\overline{1}]\\&<& [0;\tilde{\alpha}_{k},\dots, \tilde{\alpha}_{0},\tilde{\alpha}_{-1}, \dots ]+s+1\\&<&[0;\tilde{\alpha}_{k},\dots, \tilde{\alpha}_{0},\tilde{\alpha}_{-1}, \dots ]+a^*_{k+1}\\ && +[0;a^*_{k+2},a^*_{k+3}, \dots]\\&=& \lambda_0(\sigma^{k+1}(\dots,\tilde{\alpha}_{-1};\tilde{\alpha}_{0},\dots ,\tilde{\alpha}_{k},\beta_{k+1})) \\&\leq &m(\dots ,\tilde{\alpha}_{-1};\tilde{\alpha}_{0},\dots ,\tilde{\alpha}_{k},\beta_{k+1})\\&\leq&t+\epsilon/2  
\end{eqnarray*}
and for $j> k+1$, clearly
$$\lambda_0(\sigma^j(\dots,\tilde{\alpha}_{-2},\tilde{\alpha}_{-1};\tilde{\alpha}_{0},\dots ,\tilde{\alpha}_{k},s,\overline{1}))< 3 < t+\epsilon/2.$$
Then taking $x=\Pi^{-1}((\dots,\tilde{\alpha}_{-2},\tilde{\alpha}_{-1};\tilde{\alpha}_{0},\dots ,\tilde{\alpha}_{k},s,\overline{1}))$ one would have
$$x\in W^u(\tilde{\Lambda}_j)\cap W^s(\xi)\ \mbox{and}\ m_{\varphi,f}(x)\leq t+\epsilon/2+\delta/2$$
that is a contradiction. The case $a_{k+1}-1>a^*_{k+1}$ is quite similar.

Now, suppose $a_{k+1}=i$ and $a^*_{k+1}=i+1$. We affirm that $a_{k+2}=1$ because in other case by lemma \ref{lemao}, as $[0;\beta_{k+1}]<[0;i,\overline{1}]<[0;\gamma_{k+1}]$, we would have again for all $j\leq k$
\begin{eqnarray*}
    \lambda_0(\sigma^j(\dots,\tilde{\alpha}_{-2},\tilde{\alpha}_{-1};\tilde{\alpha}_{0},\dots ,\tilde{\alpha}_{k},i,\overline{1}))<  t+\epsilon/2+\delta/2.
\end{eqnarray*}
For $j> k+1$, one more time
$$\lambda_0(\sigma^j(\dots,\tilde{\alpha}_{-2},\tilde{\alpha}_{-1};\tilde{\alpha}_{0},\dots ,\tilde{\alpha}_{k},i,\overline{1}))< t+\epsilon/2$$
and for $j=k+1$,
\begin{eqnarray*}
\lambda_0(\sigma^j(\dots,\tilde{\alpha}_{-2},\tilde{\alpha}_{-1};\tilde{\alpha}_{0},\dots ,\tilde{\alpha}_{k},i,\overline{1}))&=& [0;\tilde{\alpha}_{k},\dots, \tilde{\alpha}_{0},\tilde{\alpha}_{-1}, \dots ]+i+[0;\overline{1}]\\&<& [0;\tilde{\alpha}_{k},\dots, \tilde{\alpha}_{0},\tilde{\alpha}_{-1}, \dots ]+i+1\\&<&[0;\tilde{\alpha}_{k},\dots, \tilde{\alpha}_{0},\tilde{\alpha}_{-1}, \dots ]+a^*_{k+1}\\ && +[0;a^*_{k+2},a^*_{k+3}, \dots]\\&=& \lambda_0(\sigma^{k+1}(\dots,\tilde{\alpha}_{-1};\tilde{\alpha}_{0},\dots ,\tilde{\alpha}_{k},\beta_{k+1})) \\&\leq &m(\dots ,\tilde{\alpha}_{-1};\tilde{\alpha}_{0},\dots ,\tilde{\alpha}_{k},\beta_{k+1})\\&\leq&t+\epsilon/2. 
\end{eqnarray*}
Then for $x=\Pi^{-1}((\dots,\tilde{\alpha}_{-2},\tilde{\alpha}_{-1};\tilde{\alpha}_{0},\dots ,\tilde{\alpha}_{k},i,\overline{1}))$ one would get the contradiction
$$x\in W^u(\tilde{\Lambda}_j)\cap W^s(\xi)\ \mbox{and}\ m_{\varphi,f}(x)\leq t+\epsilon/2+\delta/2.$$

Even more, we have $a_{k+3}\in \{m+1,m+2,m+3\}$ because if $a_{k+3}=\ell\leq m$, then $[0;\beta_{k+1}]<[0;i,1,\ell+1,\overline{1}]<[0;\gamma_{k+1}]$ and by lemma \ref{lemao} we would have for all $j\leq k$
\begin{eqnarray*}
    \lambda_0(\sigma^j(\dots,\tilde{\alpha}_{-2},\tilde{\alpha}_{-1};\tilde{\alpha}_{0},\dots ,\tilde{\alpha}_{k},i,1,\ell+1, \overline{1}))<t+\epsilon/2+\delta/2.
\end{eqnarray*}
For $j=k+1$,
\begin{eqnarray*}
\lambda_0(\sigma^j(\dots,\tilde{\alpha}_{-2},\tilde{\alpha}_{-1};\tilde{\alpha}_{0},\dots ,\tilde{\alpha}_{k},i,1,\ell+1,\overline{1}))&=&[0;\tilde{\alpha}_{k},\dots, \tilde{\alpha}_{0},\tilde{\alpha}_{-1}, \dots ]+i+\\&& [0;1,\ell+1,\overline{1}]\\&<&[0;\tilde{\alpha}_{k},\dots, \tilde{\alpha}_{0},\tilde{\alpha}_{-1}, \dots ]+a^*_{k+1}\\ && +[0;a^*_{k+2},a^*_{k+3}, \dots]\\&=& \lambda_0(\sigma^{k+1}(\dots,\tilde{\alpha}_{-1};\tilde{\alpha}_{0},\dots ,\tilde{\alpha}_{k},\beta_{k+1})) \\&\leq &m(\dots ,\tilde{\alpha}_{-1};\tilde{\alpha}_{0},\dots ,\tilde{\alpha}_{k},\beta_{k+1})\\&\leq&t+\epsilon/2  
\end{eqnarray*}
and for $j> k+1$, 
$$\lambda_0(\sigma^j(\dots,\tilde{\alpha}_{-2},\tilde{\alpha}_{-1};\tilde{\alpha}_{0},\dots ,\tilde{\alpha}_{k},i,1,\ell+1,\overline{1}))< m+1+ [0;\overline{1}]+[0;1,m+2,\overline{1,m+3}]< t+\epsilon/2$$
then taking $x=\Pi^{-1}((\dots,\tilde{\alpha}_{-2},\tilde{\alpha}_{-1};\tilde{\alpha}_{0},\dots ,\tilde{\alpha}_{k},i,1,\ell+1,\overline{1}))$ one would have
$$x\in W^u(\tilde{\Lambda}_j)\cap W^s(\xi)\ \mbox{and}\ m_{\varphi,f}(x)\leq t+\epsilon/2+\delta/2$$
that is again a contradiction. 

In a similar way, we must have $a^*_{k+2}\in \{m+1,m+2,m+3\}$ because if $a^*_{k+2}=\ell\leq m$, then $[0;\beta_{k+1}]<[0;i+1,\ell+1,\overline{1}]<[0;\gamma_{k+1}]$ and as before we would have for all $j\leq k$
\begin{eqnarray*}
    \lambda_0(\sigma^j(\dots,\tilde{\alpha}_{-2},\tilde{\alpha}_{-1};\tilde{\alpha}_{0},\dots ,\tilde{\alpha}_{k},i+1,\ell+1, \overline{1}))<t+\epsilon/2+\delta/2,
\end{eqnarray*}
for $j=k+1$,
\begin{eqnarray*}
\lambda_0(\sigma^j(\dots,\tilde{\alpha}_{-2},\tilde{\alpha}_{-1};\tilde{\alpha}_{0},\dots ,\tilde{\alpha}_{k},i+1,\ell+1,\overline{1}))&=&[0;\tilde{\alpha}_{k},\dots, \tilde{\alpha}_{0},\tilde{\alpha}_{-1}, \dots ]+i+1+\\&& [0;\ell+1,\overline{1}]\\&<&[0;\tilde{\alpha}_{k},\dots, \tilde{\alpha}_{0},\tilde{\alpha}_{-1}, \dots ]+a^*_{k+1}\\ && +[0;a^*_{k+2},a^*_{k+3}, \dots]\\&=& \lambda_0(\sigma^{k+1}(\dots,\tilde{\alpha}_{-1};\tilde{\alpha}_{0},\dots ,\tilde{\alpha}_{k},\beta_{k+1})) \\&\leq &m(\dots ,\tilde{\alpha}_{-1};\tilde{\alpha}_{0},\dots ,\tilde{\alpha}_{k},\beta_{k+1})\\&\leq&t+\epsilon/2  
\end{eqnarray*}
and for $j> k+1$, 
$$\lambda_0(\sigma^j(\dots,\tilde{\alpha}_{-2},\tilde{\alpha}_{-1};\tilde{\alpha}_{0},\dots ,\tilde{\alpha}_{k},i+1,\ell+1,\overline{1}))< m+1+ [0;\overline{1}]+[0;1,m+2,\overline{1,m+3}]< t+\epsilon/2$$
that let us get a contradiction again. 

In particular, in this case, we can refine the cover $\mathcal{C}_k$ by replacing the interval $I(a_0; a_1, \dots,a_k)$ with the six intervals $I(a_0;a_1, \dots, a_k,i,1,m+1),I(a_0;a_1, \dots, a_k,i,1,m+2), I(a_0;a_1, \dots, a_k,i,1,m+3), I(a_0;a_1, \dots, a_k,i+1,m+1), I(a_0;a_1, \dots, a_k,i+1,m+2)\ \mbox{and}\ I(a_0;a_1, \dots, a_k,i+1,m+3)$  for one and only one $i=1,\dots, m+2.$

Observe that, in fact, some of the intervals considered in the last paragraph, may not be possible. For example, if $\eta=m+3$ then $t+\epsilon/2<m+3$; therefore the letter $m+3$ cannot appear in the kneading sequence of any point of $\tilde{\Lambda}_j$. But this will not affect our argument. Indeed, we affirm that this procedure doesn't increase the 0.49-sum, $H_{0.49}(\mathcal{C}_k)= \sum \limits_{I\in \mathcal{C}_k} \abs{I}^{0.49}$ of the cover $\mathcal{C}_k$ of $K^u(\tilde{\Lambda}_j)$. That is, by \ref{intervals} we need to prove that
$$ \sum\limits_{j=m+1}^{m+3}\abs{I(a_1, \dots, a_k,i,1,j)}^{0.49} + \sum\limits_{j=m+1}^{m+3} \abs{I(a_1, \dots, a_k,i+1,j)}^{0.49} <  \abs{I(a_1, \dots, a_k)}^{0.49} $$
or
\begin{equation}\label{sum}
 \sum\limits_{j=m+1}^{m+3}\left(\frac{\abs{I(a_1, \dots, a_k,i,1,j)}}{\abs{I(a_1, \dots, a_k)}}\right)^{0.49}+ \sum\limits_{j=m+1}^{m+3}\left(\frac{\abs{I(a_1, \dots, a_k,i+1,j)}}{\abs{I(a_1, \dots, a_k)}}\right)^{0.49}<1
 \end{equation}
 where $i=1,\dots, m+2.$

In this direction, we have the following lemmas
 
\begin{lemma}\label{lemao}
    Given $a_0,a_1, \dots, a_n, a,b,c \in \{1, \dots, m+3 \}$ we have 
  $$\frac{\abs{I(a_1,\dots ,a_n,a,b)}}{\abs{I(a_1,\dots ,a_n)}}=\frac{1+r}{(ab+1+br)(ab+a+1+(b+1)r)}$$
  and 
$$\frac{\abs{I(a_1,\dots ,a_n,a,b,c)}}{\abs{I(a_1,\dots ,a_n)}}=\frac{1+r}{(abc+c+a+(bc+1)r)(abc+c+a+ab+1+(bc+b+1)r)}$$
where $r\in (0,1).$
\end{lemma}
\begin{proof}
  Recall that the length of $I(b_1,\dots, b_m)$ is
  $$\abs{I(b_1,\dots, b_m)}=\frac{1}{q_m(q_m+q_{m-1})},$$
where $q_s$ is the denominator of $[0;b_1,\dots, b_s]$. And that we also have the recurrence formula
$$q_{s+2}=b_{s+2}q_{s+1}+q_s.$$
Using this two and three times respectively, we have 
  $$ \abs{I(a_1,\dots ,a_n,a,b)}= \frac{1}{((ab+1)q_n+bq_{n-1})((ab+a+1)q_n+(b+1)q_{n-1})}$$
and 
\begin{align*}
  &\abs{I(a_1,\dots ,a_n,a,b,c)}\\ 
	&=\frac{1}{((abc+c+a)q_n+(bc+1)q_{n-1})((abc+c+a+ab+1)q_n+(bc+b+1)q_{n-1})}
\end{align*}
  so, we conclude 
\begin{eqnarray*}
  \frac{\abs{I(a_1,\dots ,a_n,a,b)}}{\abs{I(a_1,\dots ,a_n)}}&=&\frac{q_n(q_n+q_{n-1})}{((ab+1)q_n+bq_{n-1})((ab+a+1)q_n+(b+1)q_{n-1})}\\ &=&\frac{1+r}{(ab+1+br)(ab+a+1+(b+1)r)} 
\end{eqnarray*}
and 
\begin{align*}
  &\frac{\abs{I(a_1,\dots ,a_n,a,b,c)}}{\abs{I(a_1,\dots ,a_n)}}\\&=\frac{q_n(q_n+q_{n-1})}{((abc+c+a)q_n+(bc+1)q_{n-1})((abc+c+a+ab+1)q_n+(bc+b+1)q_{n-1})}\\ &=\frac{1+r}{(abc+c+a+(bc+1)r)(abc+c+a+ab+1+(bc+b+1)r)}
\end{align*}
with $r=\frac{q_{n-1}}{q_n}\in (0,1)$.
\end{proof}

\begin{lemma}
Fix $x,y,z,w>0$, then 
$$\frac{d}{dr}\left( \frac{1+r}{(x+yr)(z+wr)}\right)= \frac{(x-y)(z-w)-yw(r+1)^2}{(ywr^2+(xw+yz)r+xz)^2}<\frac{(x-y)(z-w)-yw}{(ywr^2+(xw+yz)r+xz)^2}$$
for $r\geq 0$.
\end{lemma}
\begin{proof}
 It's a straightforward computation.   
\end{proof}

Using the previous lemmas, that $i\geq 1$, $m\geq1$, $r\in (0,1)$, and that for $j\in \{m+1,m+2,m+3\}$ we have
$$(2j+1-(j+1))(2j+3-(j+2))-(j+1)(j+2)=j(j+1)-(j+1)(j+2)<0,$$
for the first sum one has
\begin{align*}
   & \sum\limits_{j=m+1}^{m+3}\left(\frac{\abs{I(a_1, \dots, a_k,i,1,j)}}{\abs{I(a_1, \dots, a_k)}}\right)^{0.49}\\
	&= \sum\limits_{j=m+1}^{m+3} \left(\frac{1+r}{(ij+j+i+(j+1)r)(ij+j+2i+1+(j+2)r)}\right)^{0.49}\\ 
	&\leq \sum\limits_{j=m+1}^{m+3} \left(\frac{1+r}{(2j+1+(j+1)r)(2j+3+(j+2)r)}\right)^{0.49} \\ 
	&<\sum\limits_{j=m+1}^{m+3} \left(\frac{1}{(2j+1)(2j+3)}\right)^{0.49}\\
	&\leq \left(\frac{1}{5\times7}\right)^{0.49}+\left(\frac{1}{7\times9}\right)^{0.49}+\left(\frac{1}{9\times11}\right)^{0.49}\\ 
	&< 0.412
\end{align*}
and for the second sum
\begin{align*}
&\sum\limits_{j=m+1}^{m+3}\left(\frac{\abs{I(a_1, \dots, a_k,i+1,j)}}{\abs{I(a_1, \dots, a_k)}}\right)^{0.49}\\
&= \sum\limits_{j=m+1}^{m+3} \left(\frac{1+r}{((i+1)j+1+jr)((i+1)j+i+2+(j+1)r)}\right)^{0.49} \\
 &\leq \sum\limits_{j=m+1}^{m+3} \left(\frac{1+r}{(2j+1+jr)(2j+3+(j+1)r)}\right)^{0.49}\\
 &<\sum\limits_{j=m+1}^{m+3} \left(\frac{2}{(2j+1)(2j+3)}\right)^{0.49} \\ 
&\leq
\left(\frac{2}{5\times7}\right)^{0.49}+\left(\frac{2}{7\times9}\right)^{0.49}+\left(\frac{2}{9\times11}\right)^{0.49}\\ 
&< 0.579 
\end{align*}
that proves \ref{sum} and so let us conclude that $HD(K^u(\tilde{\Lambda}_j))\leq0.49$. Finally, as we are in the conservative setting
$$HD(\tilde{\Lambda}_j)=2HD(K^u(\tilde{\Lambda}_j))<0.99.$$


Fix $\delta=\epsilon/2$. By definition, for $i\in \mathcal{I}$, $\tilde{\Lambda}_i$ connects with  $\xi$ before $t+\epsilon$, then we can apply corollary 3.4 at most $\abs{\mathcal{I}}-1$ times to see that there exists a subhorseshoe $\tilde{\Lambda}(t,\epsilon)\subset \Lambda$ and some $q(t,\epsilon)<t+\epsilon$ such that 
$$\bigcup \limits_{i\in \mathcal{I}} \tilde{\Lambda}_i\subset \tilde{\Lambda}(t,\epsilon)\subset \Lambda_{q(t,\epsilon)}.$$

Now, remember that for any subhorseshoe $\tilde{\Lambda} \subset \Lambda$ being locally maximal we have 
$$W^s(\tilde{\Lambda})= \bigcup \limits_{y\in \tilde{\Lambda}}W^s(y) \  \  \mbox{and}  \ 
 \ W^u(\tilde{\Lambda})= \bigcup \limits_{y\in \tilde{\Lambda}}W^u(y).$$
Then, for every $x\in \Lambda$ such that $\omega(x)\subset \tilde{\Lambda}$, there exists an $y\in \tilde{\Lambda}$ with \hfill\break
 $\lim \limits_{n\to\infinity}d(f(\varphi^n(x)),f(\varphi^n(y)))=0$, and so $\ell_{\varphi,f}(x)=\ell_{\varphi,f}(y)$. Using this, we have 

$$\ell_{\varphi,f}(P(t,\epsilon))=\ell_{\varphi,f}(\tilde{P}(t,\epsilon))=\bigcup \limits_{i\in \mathcal{I}} \ell_{\varphi,f}(\tilde{\Lambda}_i)\cup\bigcup \limits_{j\in \mathcal{J}} \ell_{\varphi,f}(\tilde{\Lambda}_j).$$

On the other hand
$$HD(\bigcup \limits_{j\in \mathcal{J}} \ell_{\varphi,f}(\tilde{\Lambda}_j))= \max\limits_{j\in \mathcal{J}}HD(\ell_{\varphi,f}(\tilde{\Lambda}_j))\leq \max\limits_{j\in \mathcal{J}}HD(f(\tilde{\Lambda}_j))\leq \max\limits_{j\in \mathcal{J}}HD(\tilde{\Lambda}_j)<1$$
so $int(\bigcup \limits_{j\in \mathcal{J}} \ell_{\varphi,f}(\tilde{\Lambda}_j))=\emptyset.$ 

Also, it was proved in lemma $5.2$ of \cite{GC1} that, for $\tilde{t}\leq \max f|_{\Lambda}$,
$$L\cap (-\infty,\tilde{t})= \bigcup \limits_{s<\tilde{t}} \ell_{\varphi,f}(\Lambda_s).$$ 
For the sake of completeness, we will reproduce the proof below:

  Let $x\in \Lambda$ with $\ell_{\varphi,f}(x)=\eta < \tilde t$, then there exist a sequence $\{ n_k \}_{k\in\mathbb{N}}$ such that $\lim \limits_{k \to \infinity} f(\varphi^{n_k}(x))=\eta$. By compactness, without loss of generality, we can also suppose that $\lim \limits_{k \to \infinity} \varphi^{n_k}(x)=y$ for some $y\in \Lambda$ and so that $f(y)=\eta$. 

We claim that $m_{\varphi, f}(y)=\eta$: in other case we would have for some $\tilde{k} \in \mathbb{Z}$ and $r \in \mathbb{R}$, $f(\varphi^{\tilde{k}}(y))>r>\eta$ and then for $k$ big enough by continuity $f(\varphi^{\tilde{k}+n_k}(x))>\eta$ that contradicts the definition of $\eta$. 
Now, take $\tilde{N}$ big enough such that if for two elements $a, b\in \Lambda$ their kneading sequences coincide in the central block (centered at the zero position) of size $2\tilde{N}+1$ then $\abs{f(a)-f(b)}<(\tilde t-\eta)/4$ and $N$ big enough such that for $k\geq N$ one has $f(\varphi^{k}(x))<\eta+(t-\eta)/4$ and the kneading sequences of $\varphi^{n_k}(x)$ and $y$ coincide in the central block of size $2\tilde{N}+1$. Suppose $\Pi(x)= (x_n)_{n\in \mathbb{Z}}$ and $\Pi(y)= (y_n)_{n\in \mathbb{Z}}$, then the point
$$\tilde{y}=\Pi^{-1}(\dots,y_{-n},\dots,y_{-1},x_{n_N},x_{n_N+1},\dots)$$
satisfies $\tilde{y}\in \Lambda_{(t+\eta)/2}$ and $\ell_{\varphi,f}(\tilde{y})=\eta.$ Therefore, we conclude that 
$$L\cap (-\infinity,\tilde t)=\ell_{\varphi,f}( \{x\in \Lambda: \ell_{\varphi,f}(x)<\tilde t \}) \subset \bigcup \limits_{s<\tilde t} \ell_{\varphi,f}(\Lambda_s)$$ 
and as for $s<\tilde t$, \ $\Lambda_s \subset \ell_{\varphi,f}^{-1}(-\infinity,\tilde t)$, the other inclusion also holds and we have the equality 
$$L\cap (-\infinity,\tilde t)= \bigcup \limits_{s<\tilde t} \ell_{\varphi,f}(\Lambda_s).$$

Therefore
\begin{eqnarray*}
  t\in int(m_{\varphi,f}(\Lambda_{t+\epsilon/4}))&=&int(M\cap (-\infty,t+\epsilon/4))=int(L\cap (-\infty,t+\epsilon/4))\\&=&int(\bigcup \limits_{s<t+\epsilon/4} \ell_{\varphi,f}(\Lambda_s))= int(\ell_{\varphi,f}(\Lambda_{t+\epsilon/4}))\\ &\subset& int(\ell_{\varphi,f}(P(t,\epsilon)))  
\end{eqnarray*}
and then, we must have 
$$t<\sup (\bigcup \limits_{i\in \mathcal{I}} \ell_{\varphi,f}(\tilde{\Lambda}_i))\leq \sup(\ell_{\varphi,f}(\tilde{\Lambda}(t,\epsilon)))\leq \sup f(\tilde{\Lambda}(t,\epsilon))=\max f|_{\tilde{\Lambda}(t,\epsilon)}.$$
We have then proved the following result

\begin{proposition}\label{lemmaiterativo}
Given $t\in(m+1+ [0;\overline{1}]+[0;1,m+2,\overline{1,m+3}],\eta)\cap T$ and $\epsilon<\eta-t$ there exist some $q(t,\epsilon)<t+\epsilon$ and a subhorseshoe $\tilde{\Lambda}(t,\epsilon)\subset \Lambda_{q(t,\epsilon)}$ with $HD(\tilde{\Lambda}(t,\epsilon))\geq 1$ such that
\begin{enumerate}
    \item $HD(\Lambda_t)\leq HD(\tilde{\Lambda}(t,\epsilon))$
    \item for every subhorseshoe $\tilde{\Lambda} \subset \Lambda_t$ with $HD(\tilde{\Lambda})\geq0.99$ one has $\tilde{\Lambda}\subset \tilde{\Lambda}(t,\epsilon)$
    \item $t<\max f|_{\tilde{\Lambda}(t,\epsilon)}.$
    
\end{enumerate}

\end{proposition}


\subsection{Putting unstable Cantor sets into $k^{-1}(\eta)$}

Let $\eta\in (m+1+ [0;\overline{1}]+[0;1,m+2,\overline{1,m+3}],m+4)\cap \overline{T}$ accumulated from the left by points of $T$ and $\epsilon>0$ such that $m+1+ [0;\overline{1}]+[0;1,m+2,\overline{1,m+3}]<\eta-\epsilon$. Take any strictly increasing sequence $\{t_n\}_{n\geq 0}$ of points of $T$ such that $t_0>\eta-\epsilon$ and $\lim\limits_{n\rightarrow\infty}t_n=\eta$. According to proposition \ref{lemmaiterativo}, we can find a sequence of subhorsehoes $\{\Lambda^n\}_{n\geq0}=\{\tilde{\Lambda}(t_n,(t_{n+1}-t_n)/2)\}_{n\geq0}$ with the following properties

\begin{enumerate}
    \item $HD(\Lambda_{t_n})\leq HD(\Lambda^n)$
    \item $\Lambda^n\subset \Lambda^{n+1}$
    \item $t_n<\max f|_{\Lambda^n}<t_{n+1}.$
    \end{enumerate}

Now, we will construct a homeomorphism $\theta:K^u(\Lambda^0) \rightarrow k^{-1}(\eta)$ whose inverse is H\"older inverse with exponent arbitrarily close to one.


Given $n\geq 0$, being $\Lambda^n$ a mixing horseshoe (because $\xi \subset \Lambda^n$), we can find some $c(n)\in \mathbb{N}$ such that given two letters $a$ and $b$ in the alphabet $\mathcal{A}(\Lambda^n)$ of $\Lambda^n$ there exists some finite word of size $c(n)$: $(a_1,\dots ,a_{c(n)})$ (in the letters of $\mathcal{A}(\Lambda^n)$) such that $(a,a_1,\dots ,a_{c(n)},b)$ is admissible; given $a$ and $b$ consider always a fixed $(a_1,\dots ,a_{c(n)})$ as before. Also, as $\Lambda^n$ is a subhorseshoe of $\Lambda$, it is the invariant set in the union of some rectangles determined for a set of words of size $2p(n)+1$ for some $p(n)\in \mathbb{N}$.

Now, take $n\geq 1$ and consider the kneading sequence $\{x^n_r\}_{r\in \mathbb{Z}}$ of some point $x_n\in \Lambda^n$ such that $f(x_n)=\max f|_{\Lambda^n}$. Also take $r(n)>p(n+1)+p(n)+p(n-1)$ big enough such that for any $\alpha=(a_0, a_{1} \cdots, a_{2r(n)})\in \{1,2, \cdots, m+3\}^{2r(n)+1}$ and $z,y\in R(\alpha;r(n))$ we have $\abs{f(x)-f(y)}<\min \{(t_{n+1}-\max f|_{\Lambda^n})/2,(\max f|_{\Lambda^n}-t_n)/2\}.$ Finally, set $s(n)=\sum_{k=1}^{n}(2r(k)+2c(k)+1)$.

Given $a=[a_0;a_1,a_2,\dots ]\in K^u(\Lambda^0)$ for $n\geq 1$ set $a^{(n)}:=(a_{s(n)!+1},\dots, a_{s(n+1)!})$, so one has
$$a=[a_0;a_1,a_2,\dots ]=[a_0;a_1,\dots,a_{s(1)!},a^{(1)}, a^{(2)},\dots ,a^{(n)}, \dots].$$ 

Define then
$$\theta(a):=[a_0;a_1,\dots,a_{s(1)!},h_{1}, a^{(1)},h_{2},a^{(2)},\dots ,h_{n},a^{(n)},h_{n+1}, \dots]$$
where
$$h_n=(c_1^n, x^n_{-r(n)},\dots ,x^n_{-1},x^n_0,x^n_1,\dots,x^n_{r(n)},c_2^n)$$
and $c_1^n$ and $c_2^n$ are words in the original alphabet $\mathcal{A}=\{1,\dots,m+3 \}$ with $\abs{c_1^n}=\abs{c_2^n}=c(n)$ such that $(a_0,a_1,\dots,a_{s(1)!},h_{1}, a^{(1)},h_{2},\dots ,h_{n},a^{(n)})$ appears in the kneading sequence of some point of $\Lambda^n$.

It is easy to see using the construction of $\theta$ that for every $a\in K^u(\Lambda^0)$, $k(\theta(a))=\eta$, so we have defined the map
\begin{eqnarray*}
  \theta:K^u(\Lambda^0) &\rightarrow& k^{-1}(\eta) \\
   a &\rightarrow& \theta(a)
\end{eqnarray*}
that is clearly continuous and injective.

On the other hand, given any small $\rho>0$ because of the growth of the factorial map, we have $\abs{\tilde{a}_1-\tilde{a}_2}=O(\abs{\theta(\tilde{a}_1)-\theta(\tilde{a}_2)}^{1-\rho})$ for any $\tilde{a}_1, \tilde{a}_2\in K^u(\tilde{\Lambda}_i)$ and $\abs{\tilde{a}_1-\tilde{a}_2}$ small. Indeed, if $\tilde{a}_1$ and $\tilde{a}_2$ are such that the letters in their continued fraction expressions are equal up to the $s$-th letter and $n \in \mathbb{N}$ is maximal such that $s(n)!<s$ then because $\abs{h_k}=2r(k)+2c(k)+1$; $\theta(\tilde{a}_1)$ and $\theta(\tilde{a}_2)$ coincide exactly in their first 
$$s+\sum_{k=1}^{n}(2r(k)+2c(k)+1)=s+s(n)$$
letters.

Now, let $\alpha$ and $\beta$ be finite words of positive integers bounded by $N\in \mathbb{N}$ and $I_N(\alpha)$ is the convex hull of $I(\alpha)\cap C_N$. The so-called bounded distortion property let us conclude that for some constant $C_N>1$ 
$$C_N^{-1}\abs{I_N(\alpha)}\abs{I_N(\beta)}\leq \abs{I_N(\alpha\beta)}\leq C_N\abs{I_N(\alpha)}\abs{I_N(\beta)}$$
also, for some positive constants $\lambda_1,\lambda_2<1$, one has
$$C_N^{-1} \lambda_1^{\abs{\alpha}}\leq\abs{I_N(\alpha)}\leq C_N \lambda_2^{\abs{\alpha}}$$
So, if $s$ is big such that $s(n)/(s+s(n))<\frac{\rho \log \lambda_2}{\log \lambda_1-4\log C_{m+3}}$, using lemma \ref{gugu1}, we have for some constant $\tilde{C}(m+3)$
\begin{eqnarray*}
  &&\abs{\theta(\tilde{a}_1)-\theta(\tilde{a}_2)}^{1-\rho}\\
	&\geq&\tilde{C}(m+3)^{1-\rho}\abs{I(a_1,\dots,a_{s(1)!},h_{1}, a^{(1)},\dots ,a^{(n-1)},h_{n},a_{s(n)!+1},\dots, a_s)}^{1-\rho}\\ 
	&\geq&\tilde{C}(m+3)^{1-\rho}\abs{I_{m+3}(a_1,\dots,a_{s(1)!},h_{1}, a^{(1)},\dots ,a^{(n-1)},h_{n},a_{s(n)!+1},\dots, a_s)}^{1-\rho}\\
  &=& \tilde{C}(m+3)^{1-\rho}\abs{I_{m+3}(a_1,\dots,a_{s(1)!},h_{1}, a^{(1)},\dots ,a^{(n-1)},h_{n},a_{s(n)!+1},\dots, a_s)}\times\\ 
	&& \abs{I_{m+3}(a_1,\dots,a_{s(1)!},h_{1}, a^{(1)},\dots ,a^{(n-1)},h_{n},a_{s(n)!+1},\dots, a_s)}^{-\rho}\\ 
	&\geq& \frac{1}{C_{m+3}^{2n}}\tilde{C}(m+3)^{1-\rho}\abs{I_{m+3}(a_1,\dots,a_{s(1)!})}\abs{I_{m+3}(a^{(1)})}\dots \abs{I_{m+3}(a^{(n-1)})}\times\\ 
	&&\abs{I_{m+3}(a_{s(n)!+1},\dots, a_s)}\abs{I_{m+3}(h_1)}\dots \abs{I_{m+3}(h_n)}\times\\
	&&\abs{I_{m+3}(a_1,\dots,a_{s(1)!},h_{1}, a^{(1)},\dots ,a^{(n-1)},h_{n},a_{s(n)!+1},\dots, a_s)}^{-\rho} \\ 
	&\geq&\frac{1}{C_{m+3}^{3n}}\tilde{C}(m+3)^{1-\rho}\abs{I_{m+3}(a_1,a_2,\dots,a_s)}\abs{I_{m+3}(h_1)}\dots \abs{I_{m+3}(h_n)}\times\\
	&&\abs{I_{m+3}(a_1,\dots,a_{s(1)!},h_{1}, a^{(1)},\dots a^{(n-1)},h_{n},a_{s(n)!+1},\dots, a_s)}^{-\rho}\\ 
 &\geq&\tilde{C}(m+3)^{1-\rho}\abs{I_{m+3}(a_1,a_2,\dots,a_s)}e^{-4n\log C_{m+3}}e^{s(n).\log \lambda_1}\times\\ 
&&\abs{I_{m+3}(a_1,\dots,a_{s(1)!},h_{1}, a^{(1)},\dots a^{(n-1)},h_{n},a_{s(n)!+1},\dots, a_s)}^{-\rho}\\  
&\geq&\tilde{C}(m+3)^{1-\rho}\abs{I_{m+3}(a_1,a_2,\dots,a_s)}e^{(\log \lambda_1-4\log C_{m+3})s(n)}\times\\ 
&&\abs{I_{m+3}(a_1,\dots,a_{s(1)!},h_{1}, a^{(1)},\dots a^{(n-1)},h_{n},a_{s(n)!+1},\dots, a_s)}^{-\rho}\\  
    &\geq&\tilde{C}(m+3)^{1-\rho}\abs{I_{m+3}(a_1,a_2,\dots,a_s)}e^{\rho  (s+s(n))\log \lambda_2}\times\\ 
		&&\abs{I_{m+3}(a_1,\dots,a_{s(1)!},h_{1}, a^{(1)},\dots a^{(n-1)},h_{n},a_{s(n)!+1},\dots, a_s)}^{-\rho}\\ 
  &\geq&\frac{\tilde{C}(m+3)^{1-\rho}}{C_{m+3}^{\rho}}\abs{I_{m+3}(a_1,a_2,\dots,a_s)}\times\\
  &&\abs{I_{m+3}(a_1,\dots,a_{s(1)!},h_{1}, a^{(1)},\dots ,a^{(n-1)},h_{n},a_{s(n)!+1},\dots, a_s)}^{\rho}\times\\ 
	&&\abs{I_{m+3}(a_1,\dots,a_{s(1)!},h_{1}, a^{(1)},\dots a^{(n-1)},h_{n},a_{s(n)!+1},\dots, a_s)}^{-\rho}\\ 
  &\geq&\frac{\tilde{C}(m+3)^{1-\rho}}{C_{m+3}^{\rho}}\abs{\tilde{a}_1-\tilde{a}_2}.
\end{eqnarray*}

Therefore the map $\theta^{-1}:\theta(K^u(\Lambda^0))  \rightarrow K^u(\Lambda^0)$ is a H\"older map with exponent $1-\rho$ and then 
\begin{eqnarray*}
   HD(K^u(\Lambda^0))=HD(\theta^{-1}(\theta(K^u(\Lambda^0))))&\leq& \frac1{1-\rho}HD(\theta(K^u(\Lambda^0)))\\ &\leq& \frac1{1-\rho}HD(k^{-1}(\eta)).
\end{eqnarray*}
Letting $\rho$ go to zero, we obtain
$$HD(K^u(\Lambda^0))\leq HD(k^{-1}(\eta)).$$

Now, as we indicated before, for $s\leq \max f|_{\Lambda}$ one has
 $$HD(k^{-1}(-\infty,s])=\frac{1}{2}HD(\Lambda_{s}),$$
therefore
\begin{eqnarray*}
   HD(k^{-1}(-\infty,\eta-\epsilon])=\frac{1}{2}HD(\Lambda_{\eta-\epsilon})&\leq&\frac{1}{2}HD(\Lambda_{t_0}) \leq\frac{1}{2}HD(\Lambda^0)\\ &=& HD(K^u(\Lambda^0))\leq HD(k^{-1}(\eta)). 
\end{eqnarray*}
Letting $\epsilon$ tend to zero we have
$$HD(k^{-1}(-\infty,\eta])\leq HD(k^{-1}(\eta))$$
and as the other inequality is clearly true, the first part of the result is proven for $\eta\in (m+1+ [0;\overline{1}]+[0;1,m+2,\overline{1,m+3}],m+4)\cap \overline{T}$ which is accumulated from the left by points of $T$.

For the second part of the theorem, we need the following lemma whose proof is essentially the same as the proof of lemma 2.5 of \cite{LM}. 

\begin{lemma}\label{perdadimen}
   Given $(K, \mathcal{P}, \psi)$ a $C^{\alpha}$-regular Cantor set, if $\mathcal{P}^{'} \neq \mathcal{P}$ is a finite sub collection of $\mathcal{P}$ that is also a Markov partition of $\psi$, then the Cantor set determined by $\psi$ and $\mathcal{P}^{'}$
   $$\tilde{K}=\bigcap \limits_{n\geq 0}\psi^{-n}\left( \bigcup \limits_{I\in \mathcal{P}^{'}}I \right) $$
   satisfies $HD(\tilde{K})<HD(K).$
\end{lemma}

\begin{corollary}\label{perdadimen2}
Let $\Lambda$ be a mixing horseshoe associated with a $C^2$-diffeomorphism $\varphi:S\rightarrow S$ of some surface $S$. Then for any proper mixing subhorsehoe $\tilde{\Lambda}\subset \Lambda$
    $$HD(\tilde{\Lambda})<HD(\Lambda).$$
\end{corollary}

\begin{proof}
    Refine the original Markov partition $\mathcal{P}$ of $\Lambda$ in such a way that some $\mathcal{P}^{'} \subset \mathcal{P}$,  $\mathcal{P}^{'} \neq \mathcal{P}$ is a Markov partition for $\tilde{\Lambda}$. Use the lemma \ref{perdadimen} with the maps $\psi_s$ and $\psi_u$ that define the stable and unstable Cantor sets, in order to obtain 
    $$HD(\tilde{\Lambda})=HD(K^s(\tilde{\Lambda}))+HD(K^u(\tilde{\Lambda}))<HD(K^s(\Lambda))+HD(K^u(\Lambda))=HD(\Lambda).$$
\end{proof}
Given any $t<\eta$, take $n\in \mathbb{N}$ big enough such that $t<t_n$. Now, as $\max f|_{\Lambda^n}<t_{n+1}$ and $t_{n+1}<\max f|_{\Lambda^{n+1}}$ then $\Lambda^n$ is a proper subhorseshoe of $\Lambda^{n+1}$, therefore
\begin{eqnarray*}
 HD(k^{-1}(-\infty,t])= \frac{1}{2}HD(\Lambda_t) &\leq& \frac{1}{2}HD(\Lambda^n)<\frac{1}{2}HD(\Lambda^{n+1})\\ &\leq& 
\frac{1}{2}HD(\Lambda_{t_{n+2}})\leq \frac{1}{2}HD(\Lambda_{\eta})
\\&=&HD(k^{-1}(-\infty,\eta]).   
\end{eqnarray*}
Which proves the first item of the second part of the theorem in this case.

As $m\geq 1$ was arbitrary, we have the result for $\eta \in (2+ [0;\overline{1}]+[0;1,3,\overline{1,4}],\infty)\cap \overline{T}=(3.4109...,\infty)\cap \overline{T}$ which is accumulated from the left by points of $T$. 

For $\eta \in (t^*_1,3.4109...]\cap \overline{T}$ accumulated from the left by points of $T$, consider the horseshoe $\Lambda=\Lambda(2)$ (note that $\max f|_{\Lambda(2)}=\sqrt{12}>3.4109...$). As before, given $t \in (t^*_1,\eta)\cap T$, $0<\epsilon<\eta-t$ and $\delta>0$ we consider the set 
$$\tilde{P}(t,\epsilon)=\bigcup\limits_{\substack{x\in \mathcal{X}: \ \tilde{\Lambda}_x \ is\\ subhorseshoe }}\tilde{\Lambda}_x= \bigcup \limits_{i\in \mathcal{I}} \tilde{\Lambda}_i \cup \bigcup \limits_{i\in \mathcal{J}} \tilde{\Lambda}_j$$
where for $i\in \mathcal{I}$, $\tilde{\Lambda}_i$ connects with $\xi$ before $t+\epsilon/2+\delta$ and for $j\in \mathcal{J}$, $\tilde{\Lambda}_j$ doesn't connect with $\xi$ before $t+\epsilon/2+\delta$. One more time, given $j\in \mathcal{J}$, we will suppose that there is no $x\in W^u(\tilde{\Lambda}_j)\cap W^s(\xi)$ with $m_{\varphi,f}(x)\leq t+\epsilon/2+\delta/2$.

Following the above procedure, given $k\in \mathbb{N}$ large enough we construct the coverings $\mathcal{C}_k$ of $K^u(\tilde{\Lambda}_j)$ in such a way that given $I(a_0; a_1, \dots,a_k)\in \mathcal{C}_k$, if $(a_0, a_1, \dots,a_k)$ has continuations with forced first letter $a_{k+1}$ we replace the interval $I(a_0;a_1, \dots, a_k)$ with the interval $I(a_0;a_1, \dots, a_k, a_{k+1})$. 

On the other hand, if $(a_0, a_1, \dots,a_k)$ has two continuations with different initial letter, say $ \gamma_{k+1}=(1, 
a_{k+2},\dots)$ and $\beta_{k+1}=(2, a^*_{k+2},\dots)$. Take $\alpha=(\alpha_n)_{n\in \mathbb{Z}}\in \Pi(\tilde{\Lambda}_j)$ and $\tilde{\alpha}=(\tilde{\alpha}_n)_{n\in \mathbb{Z}}\in \Pi(\tilde{\Lambda}_j)$, such that $\alpha=(\dots,\alpha_{-2},\alpha_{-1};a_0, a_1, \dots,a_k,\gamma_{k+1})$ and $\tilde{\alpha}=(\dots,\tilde{\alpha}_{-2},\tilde{\alpha}_{-1};a_0, a_1, \dots,a_k,\beta_{k+1})$. We claim that $a_{k+2}=1$ because in other case by lemma \ref{lemao}, as $[0;\beta_{k+1}]<[0;\overline{1}]<[0;\gamma_{k+1}]$, we would have for all $j\leq k$
\begin{eqnarray*}
    \lambda_0(\sigma^j(\dots,\tilde{\alpha}_{-2},\tilde{\alpha}_{-1};\tilde{\alpha}_{0},\dots ,\tilde{\alpha}_{k},\overline{1}))<  t+\epsilon/2+\delta/2,
\end{eqnarray*}
and for $j\geq k+1$
$$\lambda_0(\sigma^j(\dots,\tilde{\alpha}_{-2},\tilde{\alpha}_{-1};\tilde{\alpha}_{0},\dots ,\tilde{\alpha}_{k},\overline{1}))<3 < t+\epsilon/2.$$
Then for $x=\Pi^{-1}((\dots,\tilde{\alpha}_{-2},\tilde{\alpha}_{-1};\tilde{\alpha}_{0},\dots ,\tilde{\alpha}_{k},i,\overline{1}))$ one would get the contradiction
$$x\in W^u(\tilde{\Lambda}_j)\cap W^s(\xi)\ \mbox{and}\ m_{\varphi,f}(x)\leq t+\epsilon/2+\delta/2.$$

In a similar way, we must have $a^*_{k+2}=2$ because if $a^*_{k+2}=1$, then $[0;\beta_{k+1}]<[0;2,2,\overline{1}]<[0;\gamma_{k+1}]$ and by lemma \ref{lemao} we would have for all $j\leq k$
\begin{eqnarray*}
    \lambda_0(\sigma^j(\dots,\tilde{\alpha}_{-2},\tilde{\alpha}_{-1};\tilde{\alpha}_{0},\dots ,\tilde{\alpha}_{k},2,2, \overline{1}))<t+\epsilon/2+\delta/2,
\end{eqnarray*}
for $j=k+1$,
\begin{eqnarray*}
\lambda_0(\sigma^j(\dots,\tilde{\alpha}_{-2},\tilde{\alpha}_{-1};\tilde{\alpha}_{0},\dots ,\tilde{\alpha}_{k},2,2,\overline{1}))&=&[0;\tilde{\alpha}_{k},\dots, \tilde{\alpha}_{0},\tilde{\alpha}_{-1}, \dots ]+2+\\&& [0;2,\overline{1}]\\&<&[0;\tilde{\alpha}_{k},\dots, \tilde{\alpha}_{0},\tilde{\alpha}_{-1}, \dots ]+a^*_{k+1}\\ && +[0;a^*_{k+2},a^*_{k+3}, \dots]\\&=& \lambda_0(\sigma^{k+1}(\dots,\tilde{\alpha}_{-1};\tilde{\alpha}_{0},\dots ,\tilde{\alpha}_{k},\beta_{k+1})) \\&\leq &m(\dots ,\tilde{\alpha}_{-1};\tilde{\alpha}_{0},\dots ,\tilde{\alpha}_{k},\beta_{k+1})\\&\leq&t+\epsilon/2  
\end{eqnarray*}
and for $j> k+1$, 
$$\lambda_0(\sigma^j(\dots,\tilde{\alpha}_{-2},\tilde{\alpha}_{-1};\tilde{\alpha}_{0},\dots ,\tilde{\alpha}_{k},2,2,\overline{1}))< 2+ [0;\overline{1}]+[0;2,\overline{2,1}]=3,0406...< t+\epsilon/2$$
then taking $x=\Pi^{-1}((\dots,\tilde{\alpha}_{-2},\tilde{\alpha}_{-1};\tilde{\alpha}_{0},\dots ,\tilde{\alpha}_{k},2,2,\overline{1}))$ one would have
$$x\in W^u(\tilde{\Lambda}_j)\cap W^s(\xi)\ \mbox{and}\ m_{\varphi,f}(x)\leq t+\epsilon/2+\delta/2$$
that is again a contradiction.

In particular, in this case, we can refine the cover $\mathcal{C}_k$ by replacing the interval $I(a_0; a_1, \dots,a_k)$ with the intervals $I(a_0;a_1, \dots, a_k,1,1)$ and $I(a_0;a_1, \dots, a_k,2,2)$. 

By lemma \ref{lemao} we have the inequality

\begin{eqnarray*}
&& \left(\frac{\abs{I(a_1, \dots, a_k,1,1)}}{\abs{I(a_1, \dots, a_k)}}\right)^{0.49}+\left(\frac{\abs{I(a_1, \dots, a_k,2,2)}}{\abs{I(a_1, \dots, a_k)}}\right)^{0.49} = \\
 && \left(\frac{1+r}{(2+r)(3+2r)}\right) ^{0.49}+ \left(\frac{1+r}{(5+2r)(7+3r)}\right) ^{0.49}< \\
 && \left(\frac{2}{2\times 3}\right) ^{0.49}+  \left(\frac{2}{5\times 7}\right) ^{0.49} < 0.9 \\
\end{eqnarray*}
that let us conclude that $HD(\tilde{\Lambda}_j)<0.99$ again. The rest of the proof follows the same lines as the previous one.

Finally, if $\eta \in \overline{T}$ is accumulated from the right by points of $T$, as before, we can consider the horseshoe $\Lambda=\Lambda(N)$, where $N=2$ if $\eta\le 2+[0;\overline{1}]+[0;1,3,\overline{1,4}]$ and $N=\max\{4,\lfloor \eta \rfloor\}$ otherwise. Take any strictly decreasing sequence $\{t_n\}_{n\geq 1}$ of points of $T$ such that $\lim\limits_{n\rightarrow\infty}t_n=\eta$ and $t_1<\max f|_{\Lambda}$, also $\epsilon>0$ small enough such that $HD(\Lambda_{\eta-\epsilon})>0.99$ and take any $t_0\in(\eta-\epsilon,\eta)$. The techniques we developed allow us construct then, a sequence $\{\Lambda^n\}_{n\geq0}$ of subhorseshoes of $\Lambda$ with the following properties

\begin{enumerate}
    \item $\max f|_{\Lambda^0}<\eta$
    \item $\max f|_{\Lambda^1}<\max f|_{\Lambda}$
    \item $t_{n+1}<\max f|_{\Lambda^{n+1}}<t_n$,\ \ $\forall n\geq 1$
    \item $HD(\Lambda_{t_n})\leq HD(\Lambda^n)$,\ \ $\forall n\geq 0$
    \item $\Lambda^0\subset\Lambda^{n+1}\subset \Lambda^n$, \ \ $\forall n\geq 1.$
     \end{enumerate}
Therefore, we can define a map
\begin{eqnarray*}
  \theta:K^u(\Lambda^0) &\rightarrow& k^{-1}(\eta) \\
   a &\rightarrow& \theta(a)
\end{eqnarray*}
given by
$$\theta(a)=[a_0;a_1,\dots,a_{s(1)!},h_{1}, a^{(1)},h_{2},a^{(2)},\dots ,h_{n},a^{(n)},h_{n+1}, \dots]$$
where
$$a=[a_0;a_1,a_2,\dots ]=[a_0;a_1,\dots,a_{s(1)!},a^{(1)}, a^{(2)},\dots ,a^{(n)}, \dots]$$
and the sequences $\{s(n)\}_{n\geq1}$ and $\{h_n\}_{n\geq1}$ are defined as before and are such that $(a^{(n)},h_{n+1},a^{(n+1)},h_{n+2}, \dots)$ appears in the kneading sequence of some point of $\Lambda^{n+1}$.

It is easy to see using the construction of $\theta$ that for every $a\in K^u(\Lambda^0)$, $k(\theta(a))=\eta$ and arguing as before, that $\theta$ is
a homeomorphism with H\"older inverse whose exponent is arbitrarily close to one. That implies that $HD(k^{-1}(-\infty,\eta])=HD(k^{-1}(\eta)).$

For the second part, corollary \ref{perdadimen2} implies, one more time that $HD(\Lambda^{n+1})<HD(\Lambda^n)$, for $n\geq 1$ and then that $HD(k^{-1}(-\infty,\eta])<HD(k^{-1}(-\infty,t])$, $\forall t>\eta$, as we wanted to show.

\appendix 
\section {Hyperbolic sets of finite type} \label{appendix}

Given a horseshoe $\Lambda$, we know that there is a subshift of finite type $\Sigma_{\mathcal{B}}\subset \Sigma_{\mathcal{A}}$ and a homeomorphism $\Pi:\Lambda \rightarrow \Sigma_{\mathcal{B}}$ such that $\varphi\circ \Pi=\Pi\circ \sigma$, as explained before. Take a finite collection $X$ of finite admissible words $\alpha=(a_0,a_1,\dots,a_n)$, we said that the maximal invariant set 
$$M(X)=\bigcap \limits_{n \in \mathbb{Z}} \varphi ^{-n}(\bigcup \limits_{\alpha \in X}  R(\alpha;0))$$ 
is a \textit{hyperbolic set of finite type}. Even more, it is said to be a \textit{subhorseshoe} of $\Lambda$ if it is nonempty and $\varphi|_{M(X)}$ is  transitive. Observe that a subhorseshoe need not be a horseshoe; indeed, it could be a periodic orbit in which case it will be called a trivial subhorseshoe.

By definition, hyperbolic sets of finite type have local product structure. In fact, any hyperbolic set of finite type is a locally maximal invariant set of a neighborhood of a finite number of elements of some Markov partition of $\Lambda$:


If $X$ is as before and $n(X)=\max \limits_{\alpha \in X} \abs{\alpha}$, then the set $\Tilde{X}$ of admissible words $\Tilde{\alpha}=(a_{-n(X)}, \dots , a_0, \dots, a_{n(X)})$ such that $\Tilde{\alpha}=\alpha_1\alpha_2\alpha_3$ where the words $\alpha_1, \alpha_2, \alpha_3$ are admissible and for some $n$, $\alpha_2=(a_0,a_1,\dots,a_n)\in X$, satisfies 
$$M(X)=\bigcap \limits_{n \in \mathbb{Z}} \varphi ^{-n}(\bigcup \limits_{\tilde{\alpha} \in \tilde{X}}  R(\alpha;0)).$$ 

Suppose then, without loss of generality, that $X\subset \mathcal{A}$. We set $A=A(X)$ as the matrix with entries $A_{\alpha,\beta}$ defined by $A_{\alpha,\beta}=1$ if the letters of $X$, $\alpha$ and $\beta$ are such that $\alpha \beta$ is admissible and $A_{\alpha,\beta}=0$ otherwise.


Let $\Sigma_{X}=\left\{\underline{\alpha}=(\alpha_{n})_{n\in \mathbb{Z}}:\alpha_{n}\in X \ \text{for all} \ n\in \mathbb{Z}\right\}$ equipped with the usual shift $\sigma:\Sigma_{X}\to\Sigma_{X}$. Consider $\Sigma_A=\{\underline{\alpha}=(\alpha_n)_{n\in \mathbb{Z}}\in \Sigma_{X}: A_{\alpha_n,\alpha_{n+1}}=1 \}$, this set is closed and $\sigma$-invariant subspace of $\Sigma_{X}$. The pair $(\Sigma_{A},\sigma)$ is the two-side subshift of finite type associated to $A$ in the alphabet $X$.


Given $\alpha, \beta \in X$, we said that $\alpha$ is related to $\beta$ if for some $k,\ell >0$, $(A^k)_{\alpha, \beta}>0$ and $(A^{\ell})_{\beta, \alpha}>0$. This corresponds to having a path from $\alpha$ to $\beta$ and a path from $\beta$ to $\alpha$ in the graph $G_A$ that have as vertex set, the set $X$ and as transition matrix, the matrix $A$. We said $\alpha \in X$ is a transient state if $\alpha$ is not related to itself, i.e there is no path from it to itself. In this context, the set $\Sigma_{A}$ can be identified as the set of infinite paths in the graph $G_A$.

We said that $A$ is irreducible if for every $\alpha, \beta \in X$ there exists some $\ell \in \mathbb{N}$ with $(A^{\ell})_{\alpha, \beta}>0$. Equivalently, the matrix $A$ is irreducible when it is possible to connect by a path each pair of vertex in the graph $G_A$.

Using the above relation, we can divide $X$ into transients states and a collection of disjoint classes that determine some submatrices of $A$. More precisely, it follows from theorem 1.3.10 of \cite{BK} that there is a permutation matrix $P$ such that 
$$
P^{-1}AP=
\begin{bmatrix} 
	A_1 & * & *& \dots \ * \\
	0 & A_2 & *& \dots \ * \\
	0 & 0 & A_3 & \dots \ * \\
    \vdots & \vdots & \vdots  & \ddots \ \vdots \\
    
    0 & 0 & 0  & \dots \ A_m \\
	\end{bmatrix}
	\quad 
 $$ 
 where each $A_i$ is a transition irreducible matrix that we call of transitive component of $A$ or the one by one matrix $[0]$ corresponding to a transient state. If $A_{i_1}, \dots, A_{i_r}$ are the transitive components of $A$,   we said that $\Sigma_{A_{i_1}}, \dots, \Sigma_{A_{i_r}}\subset \Sigma_A$ are the transitive components of $\Sigma_A$. 
 
 Next, we write observation 5.1.2 of \cite{BK}
\begin{proposition}\label{irreducible}
    For $x\in \Sigma_A$, the following holds:
    \begin{itemize}
        \item the positive and negative limit sets of x, $\omega(x)$ and $\alpha(x)$ are each contained in a transitive component of $\Sigma_A$,
        \item x is nonwandering if and only if it belongs to a transitive component of $\Sigma_A$,
        \item x is nonwandering if and only if $\omega(x)\cup \alpha(x)$ is a subset of some transitive component of $\Sigma_A$.
    \end{itemize}
\end{proposition}

By proposition \ref{irreducible}, if some $x\in \Sigma_A$ doesn't belong to any transitive component of $\Sigma_A$, then x is nonwandering and there are different transitive components $\Sigma_{A_{i_a}}$ and $\Sigma_{A_{i_b}}$ such that $\omega(x) \subset \Sigma_{A_{i_a}}$ and $\alpha(x) \subset \Sigma_{A_{i_b}}$.

\begin{definition}
Any $\tau \subset M(X)$ for which there are two different subhorseshoes $\Lambda_1$ and $\Lambda_2$ of $\Lambda$ such that 
$$\tau=\{x\in M(X):\ \omega(x)\subset \Lambda_1\ \text{and}\ \alpha(x)\subset \Lambda_2  \}$$
will be called a transient set or transient component of $M(X)$.
\end{definition}

Note that by the local product structure, given a transient set $\tau$ as before,
\begin{equation}
    HD(\tau)=HD(K^s(\Lambda_2))+HD(K^u(\Lambda_1)).
\end{equation}
From the last discussion, we can recover a decomposition of the set $\Pi(M(X))$ and then for the set $M(X)$

\begin{proposition}\label{appendix}
Any hyperbolic set of finite type $M(X)$, associated with a finite collection of finite admissible words $X$, can be written as
$$M(X)=\bigcup \limits_{i\in \mathcal{I}} \tilde{\Lambda}_i $$ 
where $\mathcal{I}$ is a finite set of indices (that may be empty) and for $i\in \mathcal{I}$,\ $\tilde{\Lambda}_i$ is a subhorseshoe or a transient set.
\end{proposition}

\end{document}